\documentclass{amsart}
\usepackage{CZ}
\fullpage

\shownotes
\begin{document}

\title{A stable rank filtration on direct sum $K$-theory}
\author[Campbell]{Jonathan Campbell}
\author[Zakharevich]{Inna Zakharevich}
\author[]{Appendix by Alexander Kupers}
\date{\today}
\maketitle

\begin{abstract}
  In the literature, there are two standard rank filtrations on $K$-theory: an
  ``unstable'' one which is traditionally defined through the homology of
  $GL_n$, and a ``stable'' one which was defined by Rognes using the simplicial
  structure on Waldhausen's $S_\dotp$-construction.  In this paper we give an
  alternate stable rank filtration, which uses the simplicial structure present
  in a $\Gamma$-space construction of $K$-theory; we investigate this in the
  case of ``convenient addition categories,'' and show that in good situtations
  where a notion of ``rank'' is present, the filtration quotients will be
  homotopy coinvariants of certain highly-connected suspension spectra.  This
  approach generalizes Rognes's results on the common basis complex, and
  produces an alternate spectral sequences converging to the homology of
  algebraic $K$-theory.
\end{abstract}

\section*{Introduction}

In \cite{quillen73-fingen}, Quillen showed that the groups $K_i(R)$, for $R$ a ring of algebraic integers, are finitely generated.  To prove this he used a \emph{rank filtration}: he observed that for the category $\Mod_R$ of finitely-generated projective $R$-modules, the category $Q\Mod_R$ had a natural filtration by dimension.  He was able to compute the homologies of the associated graded pieces, and use this to show that the $K$-groups of rings of algebraic integers are finitely generated.  This computation also connected algebraic $K$-groups to the homology of Lie groups made discrete.  However, this filtration has one key failure: it is \emph{unstable}, in that it does not extend to the spectral models of algebraic $K$-theory.


In \cite{rognes}, John Rognes constructed a rank filtration on the algebraic
$K$-theory of a commutative ring in such a way that the filtration existed on
the $K$-theory \emph{spectrum}, not just on the $K$-groups or the $K$-theory
space.  Moreover, Rognes was able to give an explicit model of the filtration quotients (as spectra), and analyzes them as the homotopy orbits of general linear groups acting on certain spectra associated to decompositions of free modules.  To understand these spectra, Rognes constructed a second filtration (the
\emph{poset} filtration) in order to show that the filtered quotients of the
rank filtration are equivalent to the homotopy coinvariants of a suspension
spectrum.  These results led Rognes to be able to show that $K_4(\Z) = 0$.  Rognes described the suspension spectrum using the \emph{common basis complex}: a simplicial complex whose $k$-simplices are sets of submodules of $R^n$ for which there exists a common basis.  Although this simplicial complex has dimension $2^n-3$, Rognes showed that its homology is restricted to dimensions $n-1$ through $2n-2$.  He conjectured that this homology was actually concentrated in degree $2n-2$; this was later shown for fields (and other nice rings) in \cite{galatius_kupers_randal-williams, miller_patzt_wilson}.


In this paper we construct an analog of this filtration for situations where
we are trying to analyze the $K$-theory of a particularly nice class of symmetric monoidal categories we call \emph{convenient addition categories} (see Definition~\ref{def:convenient}). We begin by defining the notion of a filtration on a spectrum $X$ indexed over a general poset $\P$. 
Using some generalizations of Rognes's techniques we are able to show the following:
\begin{maintheorem}[See Theorem~\ref{thm:graded-part-SA}]
    Let $\A$ be a convenient addition category equipped with a functor $\nu:\A \rto \P$ for a poset $\P$; suppose further that $\nu$ reflects isomorphisms.  There is a filtration $F_p$ (indexed over $p\in \P$) on the spectrum $K(\A)$.  Let $\Decomp_A^\circ$ be the poset of nontrivial decompositions of $A$ (see Definition~\ref{def:dec}).  Then
    \[F_pK(\A)/F_{<p}K(\A) \simeq K(S_A)_{h\Aut(A)}\]
    where $A\in \A$ is such that $\nu(A) = p$, and where $S_A$ is a certain $\Gamma$-space.
\end{maintheorem}

In order to analyze $K(S_A)$ we construct a valuation on $S_A$ indexed over the partial order $\Decomp_A$ of decompositions of $A$ in $\A$ (see Definition~\ref{def:dec}) and lift this valuation to $K(S_A)$. It turns out that $K(S_A)$ satisfies the conditions of the following technical theorem: 
\begin{maintheorem}[See Corollary \ref{cor:filtered-spectrum-susp}]
    Let $\P^\circ$ be a cosieve in $\P$.  Let $X$ be a symmetric spectrum of
  simplicial sets equipped with a $\P$-valuation $\nu$, and suppose that for all $n$,
  if $A\notin \P^\circ$, then $F_{\nu(A)}X \simeq S^n$ and if
  $A \in\P^\circ$, then $F_{\nu(A)}X$ is $2n-1$-connected.  Then,
  \[X \simeq \Sigma^\infty \Sigma' N\P^\circ.\]
\end{maintheorem}

Applying the above to the case of $K(S_A)$ gives the following:
\begin{maintheorem}[See Corollary~\ref{cor:KSA}]
Let $\Decomp_A^\circ$ be the poset of nontrivial decompositions of $A$ (see Definition~\ref{def:dec}).  Then, $\Aut(A)$-equivariantly,
    \[K(S_A) \simeq \Sigma^\infty \Sigma' N \Decomp_A^\circ,\]
    where $\Sigma' X = \Sigma X$ if $X$ is nonempty, and $\Sigma \emptyset = S^0$.
\end{maintheorem}

We use these theorems to give a new presentation of Rognes's proof of the Barratt--Priddy--Quillen Theorem, as well as to give new spectral sequences converging to the rationalized $K$-groups of commutative rings $R$ and of inner product spaces over an ordered field (see Section~\ref{sec:applications}).  Moreover, we explain how to generalize this analysis to Waldhausen categories so as to recover an alternate model of Rognes's common basis complex:
\begin{maintheorem}[See Section~\ref{sec:rogneswald}]
    Let $R$ be a commutative ring, and $C_{R,n}$ be Rognes's common basis complex for $R^n$.  Let $\P^\circ$ be the poset of nontrivial minimal spanning posets (see Definition~\ref{def:spanning}).  Then
    \[\Sigma^\infty \Sigma' C_n \simeq \Sigma^\infty \Sigma' N\P^\circ.\]
\end{maintheorem}

\subsection*{Organization} This paper is organized as follows.  In
Section~\ref{sec:valuation} we introduce the notion of a $\P$-valuation on a
$\Set$-valued functor, and use it to analyze the $K$-theory of a $\Gamma$-set.
In Section~\ref{sec:convenient} we introduce convenient addition categories and
their $K$-theory, and show that a valuation of a convenient addition category
produces a valuation on its $K$-theory.  In Section~\ref{sec:subobjects} we
consider a \emph{subobject structure} and show that the filtered quotients of
certain valuations can be analyzed as homotopy coinvariants of the $K$-theories
of subobject structures; we then use a new \emph{decomposition} valuation to
show that the $K$-theory of a subobject structure is a suspension spectrum.  In
Section~\ref{sec:applications} we give the mentioned applications of these results.

\subsection*{Notation and conventions}

$\mathbf{n}\in \FinSet_*$ is the set $\{*,1,\ldots,n\}$.

$S^1$ is the pointed simplicial set with $(S^1)_n = \mathbf{n}$.  $S^n =
(S^1)^{\wedge n}$.

For a $\Gamma$-space (or $\Gamma$-set) $X$ and a simplicial pointed finite set
$L: \Delta^\op \rto \FinSet_*$, we write $X(L)$ for the simplicial space
(resp. set) $X \circ L$.

For a pointed set $S$, write $S^\circ \defeq S \backslash \{*\}$.

A \emph{preorder} is a category where $|\Hom(A,B)| \leq 1$ for all $A,B$.  A
\emph{poset} is a skeletal preorder---i.e. a preorder in which all isomorphisms
are identity morphisms.  This definition agrees (in the case of small posets)
with the more classical definition of a poset---that of a set with an ordering
placed on it.  We will sometimes write $A \leq B$ for objects $A,B$ for which
$|\Hom(A,B)| = 1$.

\subsection*{Acknowledgements} The authors would like to thank Daniel Dugger,
who kindly pointed us the proof of Theorem~\ref{thm:nconn}, and Oscar Randal--Williams and Alexander Kupers for answering the question that originally led to Appendix~\ref{app:kupers}.  Zakharevich was supported in part by NSF CAREER DMS-1846767 and by a Simons Fellowship. Kupers acknowledges the support of the Natural Sciences and Engineering Research Council of Canada (NSERC) [funding reference number 512156 and 512250], and was supported by an Alfred P.~Sloan Research Fellowship.

\section{Valuations and filtrations on $\Set$-valued functors} \label{sec:valuation}

\subsection{Basic definitions}

The definition of a category of elements and the related results are standard;
we include them in order to standardize variances and notation, as well as for
ease of reading. See for example \cite[Section 8.3]{Borceux_1994} for a more
general introduction.

\begin{definition}
  Let $F: \I \rto \Set$ be any functor.  The \emph{category of elements} of $F$
  is the category $\El_\I F$ with
  \begin{description}
  \item[objects] The objects consist of pairs $(i,x)$, with $i\in \ob\I$ and
    $x\in F(i)$.  When it is clear from context we often denote objects of
    $\El_\I F$ as simply $x$, with $i$ suppressed.
  \item[morphisms]
    \[\Hom_{\El_\I F}\big((i,x),(j,y)\big) \defeq \{\phi \in \Hom_\I(j,i) \,|\,
      F(\phi)(y) = x\}.\]
  \end{description}
  When $\I$ is clear from context, or in the special case when
  $\I = \Delta^\op$, we write simply $\El F$ instead of $\El_{\I} F$.
\end{definition}

The variance here is chosen so that when $X$ is a simplicial set, a morphism
points from faces to the simplices that they are faces of, so that morphisms
point from ``smaller'' to ``larger'' for the face maps.  This will agree with
the variance that we want once we introduce a valuation.

\begin{lemma} \label{lem:precompose-int}
  Any functor $G:\mathcal{J} \rto \I$ induces a functor $\El_{\mathcal{J}} FG
  \rto \El_\I F$.  
\end{lemma}

\begin{lemma}
  Let $\alpha: F \Rto G$ be a natural transformation of functors $\I \rto
  \Set$.  Then $\alpha$ induces a functor
  \[\El \alpha: \El F \rto \El G.\]
\end{lemma}

\begin{proof}
  For every object $i\in \I$, $\alpha_i: F(i) \rto G(i)$.  In particular, we can
  define $(\El\alpha)(i,x) = \alpha_i(x)$.  This defines $\El \alpha$ on
  objects.  Consider a morphism $(i,x) \rto (j,y)$ in $\El F$ represented by
  $\phi: j \rto i$.  We claim that this induces a morphism
  $(i,\alpha_i(x)) \rto (j, \alpha_j(y))$ in $\El G$, and we can therefore
  define $\El \alpha$ to map it to this morphism; composition will then be
  automatically respected, since in both $\El F$ and $\El G$ composition is
  defined to be composition in $\I$.  It therefore suffices to show that
  $G(\phi)(\alpha_j(y)) = \alpha_i(x)$.  But, since $\alpha$ is natural,
  \[G(\phi)(\alpha_j(y)) = \alpha_i(F(\phi)(y)) = \alpha_i(x),\]
  as desired.
\end{proof}

The goal of the next few results is to generalize the notion of a filtration of
a space indexed by the integers, to a more general notion of a filtration that
can be indexed by any poset.  Although the results that follow should not be
surprising to experts, we were unable to find them in the literature and thus
include them here.

\begin{definition}
  Given a poset $\P$, the poset $\P_{\leq A}$ (for any object $A\in \P$)
  is the full subcategory of those objects $B$ for which there exists a morphism
  $B \rto A$.  The poset $\P_{<A}$ is the full subcategory of $\P_{\leq A}$ of
  those objects not isomorphic to $A$.

  Write $\P_+$ for the poset $\P$ with a new initial object $+$ added.
\end{definition}

\begin{definition}
  Let $\P$ be a small poset and $X: \I \rto \Set$ a functor.  A
  \emph{$\P$-valuation} on $X$ is a functor $\nu:\El X \rto \P_+$ such
  that the preimage of $+$ contains only the basepoint.  When $\P$ is clear from
  context we simply call it a \emph{valuation}.

  Given a $\P$-valuation $\nu$, there is a filtration on $X$, with every object
  $p\in \P$ giving a subfunctor $F_pX$ of $X$ by
  \[F_pX(i) = \{x\in X(i)\,|\, \nu(x) \leq p\}.\] (When the valuation is not
  clear from context we will write this as $F^{\nu}_pX$.)  We write
  \[F_{<p}X = \bigcup_{q\in \P_{<p}} F_qX.\]
  
  A valuation $\nu$ gives a functor $\tilde \nu: \P \rto [\I,\Set]$ by
  $p \rgoesto F_pX$.  
\end{definition}

As a direct conssequence of Lemma~\ref{lem:precompose-int} we see that given a
$\P$-valuation on $X$ we obtain a $\P$-valuation on any functor obtained by
postcomposing with $X$.

\begin{lemma} \label{lem:compose-well-constrained}
  Let $X: \I \rto \Set$ be a functor, and let $\nu$ be a $\P$-valuation on $X$.
  For any functor $Y: \mathcal{J} \rto \I$, $\nu$ induces a $\P$-valuation
  $\nu'$ on $XY$. 
\end{lemma}

\begin{proof}
  The valuation $\nu': \El XY \rto \P$ is given by $\nu'((j,x)) =
  \nu(Y(j),x)$.
\end{proof}





Consider the case when $\I = \Delta^\op$.  In any classical filtration on a
simplicial set, it should be the case that for any simplex $x$ and any $i$,
$\nu(d_ix) \leq \nu(x)$ and that $\nu(s_ix) = \nu(x)$.  The first of these
follows from the definition of a valuation, and the same definition implies that
$\nu(s_ix) \leq \nu(x)$.  The following lemma shows that, as expected,
$\nu(s_ix) = \nu(x)$, and moreover this holds for any morphism with a left
inverse.  

\begin{lemma} \label{lem:inv-inverse}
  Suppose that $\P$ is a small poset and $X: \I \rto \Set$ is any functor.  A
  $\P$-valuation $\nu$ on $X$ is invariant across any morphism represented by a
  morphism in $\I$ with a left inverse.
\end{lemma}

\begin{proof}
  Let $\phi: i \rto j$ in $\I$ be a morphism with a left inverse
  $\psi:j \rto i$, and suppose that $\phi$ induces a morphism $f:x \rto y$ in
  $\El X$.  We have
  \[F(\psi)(x) = F(\psi)\big(F(\phi)(y)\big) = F(\psi\phi)(y) = y.\] Thus $\psi$
  represents a morphism $y \rto x$ which is a right inverse to $f$.  Then
  $\nu(f)$ has a right inverse.  But since $\P$ is a poset any morphism with a
  one-sided inverse must be the identity, so $\nu(f)$ is an identity morphism,
  as desired.
\end{proof}

As we will be studying $K$-theory functors, it is useful to be able to extend
the notion of a $\P$-valuation to work for symmetric spectra, as well as
simplicial sets.  

\begin{definition}
  Let $Y$ be any simplicial set together with a $\P$-valuation $\nu$.  For any
  simplicial set $K$, we define a $\P$-valuation $\nu\wedge K$ by $(\nu\wedge
  K)(y,k) = \nu(y)$.  
  
  Let $X$ be a symmetric spectrum of simplicial sets.

  A \emph{$\P$-valuation} $\nu_\dotp$
  on $X$ is a sequence of $\P$-valuations $\nu_n: \El X_n \rto \P$ such
  that for all $n \geq 0$ the diagram
  \begin{diagram}
    { \El (X_{n} \sma S^1) & & \El X_{n+1} \\
      & \P \\}; \to{1-1}{1-3}^{\sigma_{n}} \to{1-1}{2-2}_{\nu_{n}\wedge S^1}
    \to{1-3}{2-2}^{\nu_{n+1}}
  \end{diagram}
  commutes.  (Here, $\sigma_{n}$ is the structure map of $X$.)   
\end{definition}

An important example of $\P$-valuations on spectra arises from $\P$-valuations
on $\Gamma$-sets.  Recall that for a $\Gamma$-set $X: \FinSet_* \rto \Set$, the
$K$-theory spectrum of $X$ is defined to have $n$-th space
$X\circ S^n: \Delta^\op \rto \Set$.  Here, as mentioned in the notation section,
it is important to note that we are defining $S^n = (S^1)^{\wedge n}$, with
$S^1$ given by the standard functor $\Delta^\op \rto \FinSet_*$ taking $[n]$ to
$\mathbf{n}$.  For a more detailed analysis of this particular construction, see
\cite[top p.181]{elmendorfmandell}.

\begin{lemma} \label{lem:valtoK}
  Let $X$ be a $\Gamma$-set.  A $\P$-valuation $\nu$ on $X$ induces a $\P$-valuation
  $\nu_\dotp$ on $K(X)$.  
\end{lemma}

\begin{proof}
  By Lemma~\ref{lem:precompose-int}, a $\P$-valuation
  $\nu: \El_\Gamma X \rto \P$ induces a $\P$-valuation on $XY$.  In particular,
  $\nu$ induces valuations $\nu_n: \El X(S^n) \rto \P$.  

  To check that a valuation on $X$ induces a valuation on $K(X)$ it therefore
  suffices to check that the diagram
  \begin{diagram}
    { \El (X(S^n) \sma S^1) && \El X(S^{n+1}) \\
      & \P \\};
    \to{1-1}{1-3}^{\sigma_n} \to{1-1}{2-2}_{\nu_n} \to{1-3}{2-2}^{\nu_{n+1}}
  \end{diagram}
  commutes.  First, consider an object in $\El(X(S^n) \sma S^1)$.  This is a
  triple $(m,x,i)$ with $m \geq 0$, $x\in X(S^n)_m$ and $i\in (S^1)_m$.  The map
  $\sigma_n$ takes $(x,i)$ to the simplex $\mathrm{inc}_i(x)$, where
  $\mathrm{inc}_i: S^n \rto S^{n+1}$ is the map including $S^n$ into
  $S^{n+1}$ with constant $n+1$-st coordinate equal to $i$.  Thus to show that
  this diagram commutes on objects it suffices to show that
  $\nu_{n+1}(\mathrm{inc}_i(x)) = \nu_n(x)$.  But $\mathrm{inc}_i$ has a left
  inverse, so by Lemma~\ref{lem:inv-inverse} $\nu$ is invariant across
  $\mathrm{inc}_i$.  Since both $\nu_n$ and $\nu_{n+1}$ are induced from $\nu$,
  these must agree, as desired.  That the diagram commutes on morphisms follows
  analogously. 
%
\end{proof}

\subsection{An aside on homotopy colimits}

In this section we discuss some results about homotopy colimits which we will
use in the next section to identify the homotopy type of certain filtered
$\Gamma$-sets.  

We begin by proving a standard lemma relating the homotopy type of a space to
the homotopy types of spaces covering it.  Although this type of result is
standard, we could not find the precise version that we needed in the
literature, and thus include it here.

\begin{lemma}[Mayer--Vietoris blowup] \label{lem:MVblow} Let $X\in s\Set_*$.
  Suppose that we are given a $\P$-valuation $\nu$ on $X$.  Then the natural
  inclusion $F_AX \rcofib X$ induces a weak equivalence
  \[\hocolim_\P \tilde \nu \rwe X.\]
\end{lemma}

\begin{proof}
  We model $\hocolim_\P \tilde \nu$ as a bisimplicial set given by the functor
  $G:\Delta^\op \rto s\Set_*$ given by
  \[[m] \rgoesto^G \coprod_{p_0 \rightarrow\cdots\rightarrow p_m\in N\P} F_{p_0}X.\]
  This has a natural map to the functor $[m] \rgoesto X$; our goal is to show
  that this map is a weak equivalence after taking homotopy colimits.  Let $G_n$ be
  the simplicial set
  \[[m] \rgoesto \coprod_{p_0 \rightarrow\cdots\rightarrow p_m\in N\P}
    (F_{p_0}X)_n.\] Considering $X_n$ as a discrete simplicial set, it suffices
  to check that the map $\psi_n: G_n \rto X_n$ is a weak equivalence for all
  $n$.

  Fix $x\in X_n$, and let $G_{n,x} = \psi_n^{-1}(x)$; since $X_n$ is discrete,
  it suffices to check that the \emph{unpointed} simplicial set $G_{n,x}$ is
  contractible for all $n$ and $x$.  Let $\P_x$ be the full subcategory of those
  $A\in \P$ such that $x\in F_AX$.  Then $G_{n,x} \cong N\P_x$.  But by
  definition, $\P_x = \P_{\nu(x)/}$.  This has an initial object $\nu(x)$, and
  is therefore contractible, as desired.
\end{proof}

\begin{remark}
  This is called the \emph{Mayer--Vietoris blowup} because it is a
  generalization of the situation in which one takes a cover of $X$ by open
  sets, and takes the homotopy colimit of all finite intersections of the open
  sets.  
\end{remark}

Since homotopy colimits of symmetric spectra are levelwise, we can directly
extend this result to symmetric spectra.

\begin{corollary} \label{cor:spec-MV} Let $X$ be a spectrum equipped with a
  $\P$-valuation $\nu$.  Then there exists a weak equivalence
  \[\hocolim_\P \tilde \nu \rwe X.\]
\end{corollary}

In order to show that certain types of $K$-theories of $\Gamma$-sets are
suspension spectra, it will be useful to have a concrete model of a simplicial
join. 

\begin{definition}
  For an unpointed simplicial set $X$ and a pointed simplicial set $Y$ (with
  basepoint $y_0$) write $X \star Y$ for the pointed simplicial set $(X * Y)/(X
  * y_0)$, where $*$ is the join of simplicial sets.  Since $X * y_0 \cong CX
  \simeq *$, this is homotopy equivalent to the unpointed join $X*Y$; thus $X
  \star Y \simeq \Sigma (X \wedge Y)$, where the $\wedge$ is done with any
  choice of basepoint of $X$.

  We write $\Sigma'X$ for the space $S^0 \star X$; this is homotopy equivalent
  to $\Sigma X$ when $X$ is nonempty, and $S^0$ when $X$ is empty.
\end{definition}

\begin{lemma} \label{lem:join} For a (pointed) simplicial set $K$, write $CK$
  for a choice of factorization $K \rcofib CK \rwe *$.  For any simplicial sets
  $X$ and $Y$,
  \[(CX\times Y) \underset{X \times Y}{\cup} (X \times CY) \simeq X * Y.\]
\end{lemma}

\begin{proof}
  Consider the following diagram:
  \begin{diagram}
    { X & X \times Y & Y \\
      X\times CY & X \times Y & CX \times Y \\};
    \to{1-2}{1-1}_{pr_1} \to{1-2}{1-3}^{pr_2}
    \cofib[left]{2-2}{2-1} \cofib{2-2}{2-3}
    \we{1-1}{2-1} \we{1-3}{2-3} \eq{1-2}{2-2}
  \end{diagram}
  The homotopy pushout of the top row is $X * Y$, so the homotopy pushout of the
  bottom row is, as well.  As both maps in the bottom row are cofibrations the
  homotopy pushout of the bottom row is the ordinary pushout.  This proves the lemma.
\end{proof}

\begin{proposition} \label{prop:tojoin}
  Let $\C$ be a contractible category, let $F:\C \rto s\Set_*$ be a functor, and
  let $\C^\circ$ be the full subcategory of those $A$ such that
  $F(A) \simeq *$.  Suppose that $\C^\circ$ is a cosieve in $\C$, and suppose
  that there exists a natural transformation $\alpha:c_K \Rto F$ such that for
  all $A\not\in \C^\circ$, the map $\alpha_A$ is an equivalence.  Then
  \[\hocolim_\C F \simeq N\C^\circ \star K.\]
\end{proposition}

\begin{proof}
  Pick a factorization $K \rcofib^\iota CK \racycfib *$.  Let us first consider
  the case when $F(A) = K$ whenever $A \notin \C^\circ$, and $F(A) = CK$ for
  $A\in \C^\circ$, with the morphisms in the image of $F$ being either the
  identity or $\iota$.

  Let $k_0$ be the basepoint of $K$.
  The homotopy colimit of $F$ can be modeled as the geometric realization of the
  simplicial space $X$ whose $m$-simplices are of the form
  \[X_m = \bigvee_{i_0 \rightarrow \cdots \rightarrow i_m} F(i_0).\] Since
  $F(i_0) = K$ or $CK$, and the morphisms are all either the identity or
  $\iota$, there is a natural levelwise cofibration taking this to the
  simplicial space $Y$ whose $m$-simplices are of the form
  \[Y_m = \bigvee_{i_0 \rightarrow \cdots \rightarrow i_m} CK = (N\C)_+
    \mathrel{\tilde\boxtimes} CK, = (N\C\boxtimes CK)/(N\C \boxtimes k_0)\]
  where $\boxtimes$ is the external product.  Every simplex in $X$ lies either
  inside $(N\C\boxtimes K)/(N\C \boxtimes k_0)$ or inside
  $(N\C^\circ\boxtimes CK)/(N\C^\circ \boxtimes k_0)$.  (Note that if
  $i_0 \in \C^\circ$ then so are all of the other $i$'s, since $\C^\circ$ is a
  cosieve in $\C$, so any simplex with $F(i_0) = CK$ is indexed by a simplex in
  $N\C^\circ$.)  These subspaces overlap in
  $(N\C^\circ \boxtimes K)/(N\C^\circ \boxtimes k_0)$.  Thus
  \[X \cong \big((N\C \boxtimes K) \cup_{N\C^\circ \boxtimes K} (N \C^\circ
    \boxtimes CK)\big) / ((N\C \boxtimes k_0) \cup_{N\C^\circ \boxtimes k_0} (N\C^\circ \boxtimes k_0)).\]
  Since $\C$ is contractible, $N\C \simeq k_0 \simeq C N \C^\circ$.  Taking
  geometric realizations of both sides and applying Lemma~\ref{lem:join} gives
  \[|X| \cong (N\C^\circ * K)/(N\C^\circ * k_0) = N\C^\circ \star K.\]
\end{proof}

\begin{corollary}
  Since homotopy colimits of spectra of simplicial sets are computed levelwise,
  the above proposition holds for categories into spectra of simplicial sets, as
  well.   
\end{corollary}

Suppose that $f: X_\dotp \rto Y_\dotp$ is a map of symmetric spectra.  A map of
symmetric spectra inducing isomorphism of naive stable homotopy groups is a
stable equivalence (by \cite[Theorem I.4.23]{schwedebook}).  This means that if
for all $n$ we can find an $M$ such that for all $m>M$, $f_m$ induces an
isomorphism on $\pi_{n+m}$, then $f$ is a stable equivalence.  Thus all that is
necessary for $f$ to be a stable equivalence is for the number of homotopy
groups on which $f_m$ induces isomorphisms to grow sufficiantly faster than $m$.
It can therefore be possible to construct good Mayer--Vietoris approximations
even when they are not strictly equivalences.  The following theorem is
well-known to experts, but we include it for ease of reading.  We thank Daniel
Dugger for pointing out the correct reference.

\begin{theorem} \label{thm:nconn}
  Fix a positive integer $n$.  Let $\C$ be a small category,
  $F,G: \C \to s\Set_*$ two functors, and $\alpha: F \Rto G$ a natural
  transformation such that for all $C\in \C$, $\alpha_C$ is an isomorphism on
  $\pi_i$ for all $i < n$.  Then the induced map
  \[\hocolim_\C F \rto \hocolim_\C G\]
  is also an isomorphism on $\pi_i$ for all $i < n$.
\end{theorem}

\begin{proof}
  Let $L_n s\Set_*$ be the left Bousfield localization of $s\Set_*$ with respect
  to the map $S^n \rto *$; it exists since $s\Set_*$ is combinatorial, left
  proper, and we are only localizing with respect to one map \cite[Theorem
  2.11]{barwick07}.  A morphism is a weak equivalence in $L_n s\Set_*$ if and
  only if it induces an isomorphism on $\pi_i$ for $i < n$.  Applying
  \cite[Proposition D.2]{farjoun96}, which states that a map between homotopy
  colimits which is induced by pointwise $L_n$-equivalences is itself an
  $L_n$-equivalence, implies the result.
\end{proof}

We will need the following technical lemma:
\begin{lemma} \label{lem:m-conn} Let $\P^\circ$ be a cosieve in $\P$.  Let $X$
  be a pointed simplicial set with a $\P$-valuation.  Suppose that there exists
  a pointed simplicial set $K$, and that the functor $\tilde \nu$ receives a
  natural transformation $\alpha: c_K \Rto \tilde \nu$ (where $c_K$ is the
  constant functor at $K$) such that
  \begin{itemize}
  \item when $A\notin \P^\circ$, the map $\alpha_A$ is a weak equivalence, and
  \item when $A\in \P^\circ$, the space $F_AX$ is $m$-connected.
  \end{itemize}
  Then there exists a functor $G: \P \rto sSet_*$ and a natural zig-zag of maps
  \[\hocolim_\P \tilde \nu \lwe \hocolim_\P G \rto N\P^\circ \star K\]
  where the first map is a weak equivalence and the second is an isomoprhism on
  $\pi_i$ for $i<m$.  
\end{lemma}

\begin{proof}
  Let $G: \P \rto s\Set_*$ be the functor defined by $G(A) = K$ for $A \notin
  \P^\circ$, and $G(A) = \tilde\nu(A)$ for $A\in \P^\circ$.  Then the natural
  transformation $G \Rto \nu$ (given by $\alpha$ on the portion not in
  $\P^\circ$) is a levelwise weak equivalence, and thus induces an equivalence
  on the homotopy colimit.
  
  Let $CX$ be any model of the cone on $X$.  Note that there is a natural
  transformation $\tilde \nu \rto c_{CX}$.  Let $F: \P \rto s\Set_*$ be the
  functor taking $A\notin\P^\circ$ to $K$ and $A\in \P^\circ$ to $CX$ (that
  this is well-defined follows because $\P^\circ$ is a cosieve).  Then there is
  a natural transformation $G \rto F$ which is an isomorphism on $\pi_i$ for
  $i < m$ for all $A\in \P$, which by Theorem~\ref{thm:nconn} implies that
  $\hocolim_\P \tilde \nu \rto \hocolim_\P F$ is also an isomoprhism on $\pi_i$
  for $i <m$.  $\hocolim_\P F \simeq N\P^\circ \star K$ by
  Proposition~\ref{prop:tojoin}, as desired.
\end{proof}

\subsection{Applications to spectra and $\Gamma$-sets with valuations}

The key technical result for analyzing the homotopy type of the filtered spectra
that appear in our filtrations will be the following:
\begin{corollary} \label{cor:filtered-spectrum-susp}
  Let $\P^\circ$ be a cosieve in $\P$.  Let $X$ be a symmetric spectrum of
  simplicial sets equipped with a $\P$-valuation, and suppose that for all $n$,
  if $A\notin \P^\circ$, then $\tilde \nu_n(A) \simeq S^n$ and if
  $A \in\P^\circ$, then $\tilde \nu_n(A)$ is $2n-1$-connected.  Then,
  \[X \simeq \Sigma^\infty \Sigma' N\P^\circ.\]
\end{corollary}

\begin{proof}
  Let $\widehat X = \hocolim_\P \tilde \nu$, where $\tilde \nu(A)$ is the
  spectrum with $\tilde \nu_n(A)$ at level $n$.  (This gives a well-defined
  spectrum by the assumption that $X$ is equipped with a $\P$-valuation.)  By
  Lemma~\ref{lem:MVblow}, since homotopy colimits of spectra are pointwise,
  $X \simeq \hocolim_\P \tilde \nu$.  The maps constructed in
  Lemma~\ref{lem:m-conn} assemble into a spectral map
  $\hocolim_\P G_n \rto \Sigma^n \Sigma' N\P$.  Since the $n$-th map induces
  isomorphisms up to $\pi_{2n-1}$, these form a stable equivalence, as desired.
\end{proof}

This result can also be extended naturally to filtrations on $K$-theory of
$\Gamma$-sets.
\begin{corollary} \label{cor:gamma-space-indec} Let $\P^\circ$ be a cosieve in
  $\P$.  Let $X: \FinSet_* \rto \Set$ be a $\Gamma$-set equipped with a
  $\P$-valuation $\nu$, and suppose that for all $n$, if $A\notin \P^\circ$ then
  $F_AX(S^n)\simeq S^n$ and if $A\in\P^\circ$ then $F_AX(L)$ is $2m+1$-connected
  for any $m$-connected $L\in s\FinSet_*$ (assuming $m \geq 1$).  Then
  \[K(X) \simeq \Sigma^\infty \Sigma' N\P^\circ.\]
\end{corollary}

\begin{proof}
  By Lemma~\ref{lem:valtoK}, a $\P$-valuation on $X$ induces a $\P$-valuation on
  $K(X)$.  Moreover, since $K(X)_n = X(S^n)$, it is the case that for
  $A\notin\P^\circ$, $F_AK(X) \simeq \S$, and for $A\in \P^\circ$,
  $F_AK(X) \simeq *$.  Thus, by Corollary~\ref{cor:filtered-spectrum-susp} we
  have
  \[K(X) \simeq \hocolim_\P \tilde \nu \simeq \Sigma^\infty \Sigma' N\P^\circ,\]
  as desired.
\end{proof}

\subsection{A note about $G$-equivariance}

In Section~\ref{sec:subobjects} we will want to work in spaces (and spectra)
equipped with a $G$-action for some discrete group $G$.  All of the results of
the previous section work if we replace ordinary functors $F: \I \rto \Set$ with
functors $F: \I \rto G\Set$ with values in the category of $G$-sets.  If
$F:\I \rto G\Set$ is a functor then $\El_\I F$ is a $G$-category.

\begin{definition}
  Given a functor $F: \I \rto G\Set$, and a $G$-poset $\P$, a $G$-equivariant
  $\P$-valuation is an equivariant functor $\nu:\El_\I F \rto \P_+$ (where $G$
  acts on the new initial object $+$ trivially).
\end{definition}

They key result we will need is the following:
\begin{corollary} \label{cor:Gsusp}
  Let $\P^\circ$ be a cosieve in $\P$ such that $\P^\circ$ is itself a
  $G$-category.  Let $X: \FinSet_* \rto G\Set$ be a $G$-equivariant $\Gamma$-set
  equipped with a $G$-equivariant $\P$-valuation $\nu$, and suppose that for all
  $n$, if $A \notin \P^\circ$ then there is a $G$-equivariant equivalence
  $F_AX(S^n) \simeq S^n$ (where $G$ acts trivially on $S^n$) and if $A\in
  \P^\circ$ then $F_A X(L)$ is $2m+1$-connected for any $m$-connected $L\in
  sG\FinSet_*$ (assuming $m \geq 1$).  Then 
  \[K(X) \simeq \Sigma^\infty \Sigma' N\P^\circ\]
  $G$-equivariantly.
\end{corollary}

\section{Convenient addition categories and their $K$-theory} \label{sec:convenient}

In this section we introduce convenient addition categories and their
$K$-theory.  We describe the $K$-theory in two ways: via the
$U_\dotp$-construction, which is analogous to Waldhausen's
$S_\dotp$-construction, and via a $\Gamma$-space.  We will not use the
$U_\dotp$-construction in later parts of the paper, but we include it as it is
sometimes easier for those familiar with the $S_\dotp$-construction to
understand, and it is isomorphic to the $\Gamma$-space construction in relevant
parts (as shown in Proposition~\ref{prop:DtoU}).

\subsection{Basic definitions and examples}

In this section we define the basic object of study of this paper: convenient
addition categories.  This class of categories is one in which addition is
``convenient'': it is a symmetric monoidal structure which also has some of the
properties of a pushout.

\begin{definition} \label{def:convenient}
  A \emph{convenient addition category} is a symmetric monoidal category
  $(\A,\oplus, \initial)$ such that the following conditions hold:
  \begin{itemize}
  \item[(CA1)] $\initial$ is initial in $\A$ and all morphisms in $\A$ are
    monic.
  \item[(CA2)] A morphism $A\oplus B \rto C$ is uniquely determined by the induced
    morphisms $A \rcofib A\oplus B \rto C$ and $B \rcofib A\oplus B \rto C$.
  \item[(CA3)] For all $A,B\in \A$, $A\times_{A\oplus B} B = \initial$.
  \end{itemize}
\end{definition}

Many symmetric monoidal categories produce convenient addition categories when
restricted to the subcategory of monomorphisms.  The idea of convenient addition
categories is that instead of restricting to isomorphisms in order to take the
$K$-theory of a symmetric monoidal categories, we can instead restrict to the
subcategory of monomorphisms (or a nice subcategory thereof) in order to retain
some extra tools for analysis. The $U_\dotp$-construction, introduced in the
next section, is from this perspective simply a model for the bar construction.  


\begin{example} \label{ex:finset}
  The category $\FinSet^{inj}$ of finite sets, injections, and disjoint union is a
  convenient addition category.
\end{example}

\begin{example} Let $k$ be an ordered field. The category $\Inn_k$ of
  finite-dimensional inner product spaces, isometric inclusions, and orthogonal
  sums is convenient addition category.
\end{example}

\begin{example}
  Let $\E$ be an exact category, with $m\E$ the subcategory containing all
  admissible monomorphisms.  We claim that $m\E$ is a convenient addition
  category if we define $\oplus$ to be the coproduct in $\E$.  Then (CA1) holds
  because in any category with a zero object, if a morphism $X \rto 0$ is monic
  then it is an isomorphism.  Let $\oplus$ be the coproduct in $\E$.  (CA2)
  holds because any monoidal structure determined by coproduct (possible in a
  larger category) satisfies this property. (CA3) holds because it's true in the
  category of $R$-modules of any ring $R$, and any exact category is a
  subcategory of modules.

  In particular, for a commutative ring $R$, the category $\Mod_R$ of free
  finite rank $R$-modules and injective linear maps, together with direct
  sum of modules, is a convenient addition category.  In the case when $R$ is
  field we will write $\Vect_R$. 
\end{example}

\begin{non-example}
  Consider $\Z_{\geq 0}$ as a poset, with $\initial = 0$, and define $\oplus$ to
  be addition.  Then (CA1) and (C2)
  hold, but (CA3) does not hold.

  More generally let $\C$ be a poset with an initial object and closed under
  coproducts.  Define $\oplus$ to be the pushout.  Then (CA1) holds by
  definition, and (CA2) holds because they hold in any poset.  However, there
  is no reason for (CA3) to hold.  For example, in the poset consisting of a
  single commutative square plus a new initial object, (CA3) does not hold.
\end{non-example}

A more general example illustrating a general method of constructing convenient
addition categories from symmetric monoidal categories is the following:

\begin{example}
  Let $\C$ be any category closed under pushouts which has an initial object
  $\initial$.  Moreover, suppose that in $\C$, pushout squares are also pullback
  squares, and pushouts preserve monomorphisms.  Then the subcategory of
  monomorphisms of $\C$, with $\oplus$ defined as the pushout in $\C$, is a
  convenient addition category.  Example~\ref{ex:finset} is an example of such
  as structure.
\end{example}

\subsection{The $K$-theory of a convenient addition category and the $U_\dotp$-construction}

We give two models for the $K$-theory of a convenient addition category.  The
first of these is based off of Waldhausen's $S_\dotp$-construction; although
this is clearly well-known to experts, we present it in detail so that the
analogy to the $S_\dotp$-construction is clear.  It's closely related to the iterated bar construction (see for example \cite{kadeishvili98, fresse11}) and other $E_k$-constructions such as \cite[Definition 17.17]{galatius_kupers_randal-williams18}.   In the next subsection we will
give a different presentation of this same $K$-theory using a $\Gamma$-space.

\begin{definition}
  Let $(\A,\oplus,\initial)$ be a convenient addition category.  We define the
  category $U_{n}\A$ to have
  \begin{description}
  \item[objects] Tuples
    $(A_1,\ldots,A_n, \{(S_{ij},\gamma_{ij})\}_{1 \leq i < j \leq n} )$, where
    $A_i\in \ob \A$ and $\gamma_{ij}: S_{ij} \rto^{\cong} \bigoplus_{k=i}^j
    A_k$.  We often write $(A,S,\gamma)_n$ for such an object in $U_n\A$. 
  \item[morphisms] A morphism
    $(A_1,\ldots,A_n,\{(S_{ij},\gamma_{ij})\}) \rto (A_1', \ldots,A_n',
    \{(S'_{ij},\gamma_{ij})\})$ is a tuple of morphisms $\phi_i: A_i \rto A_i'$.
    Note that the $\phi_i$ induce, for all $i < j$, unique morphisms
    $\phi_{ij}: S_{ij} \rto S'_{ij}$ making the diagram
    \begin{diagram}[5em]
      { S_{ij} & S'_{ij} \\
        \bigoplus_{k=i}^j A_k & \bigoplus_{k=i}^j A'_k\\};
      \arrowsquare{\phi_{ij}}{\gamma_{ij}}{\gamma'_{ij}}{\bigoplus_{k=i}^j \phi_k}
    \end{diagram}
    commute.
  \end{description}
  The category $U_n\A$ is itself symmetric monoidal with the monoidal structure
  given pointwise, as it is equivalent to the category $\A^n$.

  The category $U_0\A$ is the trivial category with only one object $(I, \{\})$.

  For all $i<j<k$ we write $\gamma_{ijk}: S_{ij} \rto S_{ik}$ for the induced
  morphism
  \[S_{ij} \rto^{\gamma_{ij}} \bigoplus_{\alpha=i}^j A_\alpha
    \setlen{13em}{\rto^{\bigoplus_{\alpha=i}^j 1_{A_\alpha} \oplus \bigoplus_{\alpha=j+1}^k
      (\initial \rightarrow A_\alpha)} \bigoplus_{\alpha=i}^k A_\alpha}
    \rto^{\gamma_{ik}^{-1}} S_{ik}.\]
\end{definition}

The categories $U_n\A$ fit into a simplicial category, where $d_k: U_n \A \rto
U_{n-1}\A$ are defined for $0 < k < n$ by
\[d_k(A_1,\ldots,A_n,\{S_{ij}\}) = (A_1,\ldots,S_{k(k+1)},A_{k+2},\ldots,A_n,
  \{(S_{ab}',\gamma_{ab}')\}),\]
where
\[S'_{ab} =
  \begin{cases}
    S_{ab} & b < k, \\
    S_{(a+1)(b+1)} & a > k, \\
    S_{a(b+1)} & a \leq k \leq b,
  \end{cases}
\]
and $\gamma'_{ab}$ is defined analogously.
We also define $d_0: U_n \A \rto U_{n-1}\A$ by
\[d_0(A_1,\ldots,A_n, \{(S_{ij},\gamma_{ij})\}) = (A_2,\ldots,A_n,
  \{(S_{(i+1)(j+1)},\gamma_{(i+1)(j+1)}\})\] and
\[d_n(A_1,\ldots,A_n,\{(S_{ij},\gamma_{ij})\}) = (A_1,\ldots,A_{n-1},\{(S_{ij},\gamma_{ij})\}).\]
Degeneracies $s_i$ are defined by adding an $I$ to the $i$-th slot and choosing
the $S_{ab}$ to disregard it.

\begin{remark} \label{rem:U.bar}
  The $U_\dotp$-construction is simply a model for the bar construction in the
  case where the sum $\oplus$ is not strictly associative.  Thus it is the case
  that $|iU_\dotp\A| \simeq B|i\A|$.
\end{remark}

Analogously to the case of Waldhausen categories and the $S_\dotp$-construction,
$U_\dotp\A$ is a simplicial symmetric monoidal category with the monoidal
structure induced by pointwise $\oplus$, so the construction can be iterated.

\begin{proposition}
  \[|NiU_\dotp\A| \simeq \Omega |NiU_\dotp U_\dotp\A|.\]
\end{proposition}

\begin{proof}
  (Following \cite[Proposition 1.5.5]{waldhausen}, as usual.)  We show that the
  sequence
  \[|NiU_\dotp \A| \rto |NiU_\dotp U_{dotp+1}\A| \rto |NiU_\dotp U_\dotp\A|\] is
  a homotopy fiber sequence.  From this the proposition will follow, as
  $|NiU_0| = *$, the middle object is contractible and therefore
  $|Ni U_\dotp\A| \simeq \Omega |NiU_\dotp U_\dotp\A|$.

  The sequence is defined in the following manner.  Using the fact that
  $U_1\A \cong \A$, we use the inclusion
  $U_\dotp U_1\A \rto U_\dotp U_{\dotp+1}\A$ to define the first map.  The
  second map is defined by the map $d_0: U_{\dotp+1}\A \rto U_\dotp \A$.  In the
  following we drop the $|Ni-|$, as well as the choices of $\gamma$-maps, to
  clean up the notation.  Following the referenced proof, it suffices to check
  that the sequence
  \[U_\dotp \A \rto U_\dotp U_{n+1}\A \rto U_\dotp U_n \A\]
  is a homotopy fiber sequence.  We define a map 
  \[U_k \A \times U_k U_n \A \rto U_k U_{n+1}\A\]
  on objects by
  \[(A_1,\ldots,A_k, \{S_{ij}\}) \times ((B_{j\ell})_{1\leq j \leq k, 1 \leq
      \ell \leq n}, \{T_{ij}\}) \rgoesto ((C_{j\ell})_{1\leq j \leq k, 1 \leq
      \ell \leq n+1}, \{V_{ij}\})\]
  where
  \[C_{j\ell} =
    \begin{cases}
      A_j \caseif \ell=1 \\
      B_{j(\ell+1)} \caseotherwise.
    \end{cases}\] and $V_{ij}$ is an $n+1$-tuple where the first entry is
  $S_{ij}$ and the last $n$ entries come from $T_{ij}$.  As morphisms are
  pointwise, if we extend this pointwise to morphisms it gives a functor.
  Moreover, for all $k$ it is full, faithful, and essentially surjective, since
  all objects associated to the same tuple of objects are canonically
  isomorphic.  Thus after applying $|Ni-|$ it is a weak equivalence.  

  Consider the following diagram:
  \begin{diagram}
    { U_\dotp \A & U_\dotp \A \times U_\dotp U_n \A & U_\dotp U_n \A \\
      U_\dotp \A & U_\dotp U_{n+1} \A & U_\dotp U_n \A \\};
    \to{1-1}{1-2} \to{1-2}{1-3}
    \to{2-1}{2-2} \to{2-2}{2-3}
    \eq{1-1}{2-1} \eq{1-3}{2-3}
    \to{1-2}{2-2}
  \end{diagram}
  The left-hand square commutes because with either composition it takes an
  object $A$ to the $n+1$-tuple $(A,I,\ldots,I)$.  The right-hand square
  commutes because it takes a pair
  $(A_1,\ldots,A_k, S_{ij}) \times ((B_{j\ell}), T_{ij})$ to
  $((B_{j\ell}), T_{ij})$, either by first assembling into a single object and
  then dropping the first column, or else by just projecting.  The top row is
  trivially a fibration sequence.  As the middle map is a weak equivalence, the
  bottom row is therefore a homotopy fibration sequence, as desired.
\end{proof}

In an analogous fashion we can prove the following:
\begin{corollary}
  For $k\geq 1$,
  \[|NiU_\dotp^{(k)}\A| \simeq \Omega |Ni U_\dotp^{(k+1)}\A|.\]
\end{corollary}

Define the spectrum $U(\A)$ to have
\[U(\A)_k = |Ni U_\dotp^{(k)}\A|.\] Then $U(\A)$ is an $\Omega$-spectrum above
level $0$.  We can think of $U^{(n)}_{\dotp}(\A)$ as containing $n$-dimensional
grids of objects in $\A$, together with choices for all sums of
``$n$-dimensional subgrids.''  

\begin{theorem} \label{thm:grpcomp}
  \[\Omega^\infty U(\A) \simeq \Omega B|i\A|.\]
\end{theorem}

\begin{proof}
  Since $U(\A)$ is an $\Omega$-spectrum above level $0$, it suffices to check
  that $|iU_\dotp\A| \simeq B|i\A|$.  Let $\D$ be the monoidal strictification
  of $\A$; then we can make a model of $B|i\A|$ by applying the bar construction
  to $i\D$ relative to this monoidal structure; this will have, at level $n$,
  $\D^n$.  Since $U_\dotp$ is functorial, $|iU_\dotp\A|\simeq |iU_\dotp
  \D|$. There is a functor $\D^n \rto U_n\D$ given by choosing the objects
  $S_{ij}$ using this strictly associative monoidal structure.  As this is full,
  faithful, and essentially surjective, this is an equivalence of categories.
  It is also a simplicial functor, and thus induces an equivalence on
  realizations.  Thus we have an equivalence
  \[|iU_\dotp\A| \simeq |iU_\dotp\D| \simeq B(*,i\D,*) = B|i\D| \simeq B|i\A|,\]
  as desired.
\end{proof}

\begin{example}
  Let $R$ be a ring, and $\Proj_R$ the category of finitely-generated projective
  $R$-modules.  Then $|iU_\dotp\A| \simeq B |i\A|$, and thus
  $\Omega |iU_\dotp\A| \simeq \Omega B |i\A|$.  Thus in particular
  $U(\Proj_R) \simeq K(R)$. 

  By an abuse of notation, we write $U(R)$ for $U(\Proj_R)$.

  Analogously we can consider $U(\Free_R)$, whose homotopy groups will agree
  with those of $U(\Proj_R)$ apart from at $\pi_0$.
\end{example}

\begin{example}
  Let $F$ be an ordered field.  Write $\Orth_F$ for the category of
  finite-dimensional inner product spaces over $F$ and isometric injections,
  with the monoidal structure given by orthogonal sums.  The $U(\Orth_F)$ will
  be the group completion of $|i\Orth_F|$ which, as discussed in \cite[Theorem 2.2]{eberhardt22} (and \cite[Theorem IV.4.9]{weibel_kbook}, will be $\Z \times BO(F)^+$.
\end{example}

For later reference it will be useful to give some ``all at once'' notation for
objects in $U_\dotp^{(n)}(\A)$.

\begin{definition}
  For an object $X = (A,S,\gamma)_n \in U_n(\A)$, write $X_i \defeq A_i$.
  Define
  \[\max X \defeq
    S_{1n}.\] For each $i\leq j$ there is an induced map
  $\phi_i: X_i \rto \max X$.
  
  For an object $X \in U_{k_1} \cdots U_{k_n}(\A)$, 

  For $n > 1$, and $X = (A,S,\gamma)_{k_1}\in \in U_{k_1}\cdots U_{k_n}(\A)$,
  write $X_{i_1\cdots i_n}$ for the object $(X_{i_1})_{i_2\cdots i_n}$.  If we
  think of $X$ as a $k_1 \times\cdots\times k_n$ grid of objects, this is simply
  pulling out the $(i_1,\ldots,i_n)$-th coordinate.  Also, let 
  \[\max X \defeq \max S_{1n}. \]
  Informally, $\max X$ is the sum of all of the $A$'s in all of the coordinates
  of $X$.  Then there is an induced map
  $\phi_{i_1\cdots i_n}: X_{i_1\cdots i_k} \rto \max X$.
\end{definition}

\subsection{$K$-theory of a convenient addition category via a special
  $\Gamma$-category}

In this section we give a different construction of the $K$-theory of a
convenient addition category $\A$ using a $\Gamma$-space.  This definition was clearly understood by Segal in \cite{segal74}, and is very closely related to the construction described in \cite[Section 4]{elmendorfmandell}, and will therefore be very familiar to experts in the field.  However, as we could not find a description for the general symmetric monoidal case given in the literature, we describe it here.

\begin{definition}
  For any $A,B\in \A$, write $\inc_{A,B}: A \rto A\oplus B$ for the morphism $A
  \cong A \oplus \initial \rto A \oplus B$.  By an abuse of notation we will
  sometimes also write $\inc_{A,B}$ for the morphism $A \rto B\oplus A$ given
  by composing $\inc_{A,B}$ with the symmetry.  
\end{definition}

\begin{definition}
  For any finite pointed set $S$, let $\I_S$ be the category subsets of $S$
  which do \emph{not} include the basepoint, and inclusions.  Given a function
  $f: S \rto T$ there is a functor $f^*:\I_T \rto \I_S$ given by
  $B \rgoesto f^{-1}(B)$; note that since $B$ does not include the basepoint,
  $f^{-1}(B)$ will also not include the basepoint, so this is well-defined.

  Let $\A$ be a convenient addition category. A functor $F: \I_S \rto \A$ is an
  \emph{addition functor} if $F(\emptyset) = \initial$, and for every disjoint
  $T,T'\in \I_S$, there exists an isomorphism
  $\phi_{T,T'}:F(T) \oplus F(T') \rto F(T \cup T')$ (which is unique by (CA2))
  making the diagram
  \begin{diagram}
    { F(T)  \\
      F(T) \oplus F(T')  & F(T \cup T')\\
      F(T')\\};
    \to{1-1}{2-1}_{\inc_{F(T),F(T')}} \to{3-1}{2-1}^{\inc_{F(T'),F(T)}} \to{2-1}{2-2}^\phi
    \to{1-1}{2-2}^{F(T \subseteq T\cup T')} \to{3-1}{2-2}_{F(T' \subseteq T \cup T')}
  \end{diagram}
  commute.  (If $\cup$ were defined for all pairs of morphisms, this would
  simply be the statement that $F$ is strong symmetric monoidal.)
\end{definition}

\begin{definition}
  Let $\A$ be a convenient addition category.  Define $D_\A: \FinSet_* \rto
  \Cat$ to send a set $S$ to the category of addition functors $\I_S \rto \A$
  and natural isomorphisms, which we denote by $[\I_S,\A]_{add}$.  A function
  $f: S \rto T$ induces a functor $[\I_S,\A]_{add} \rto [\I_T,\A]_{add}$ by
  precomposition with $f^*$.
\end{definition}

\begin{lemma}
  $D_\A$ is well-defined and is a special $\Gamma$-category.
\end{lemma}

\begin{proof}
  To check that $D_\A$ is well-defined we must check that for all $f: S \rto T$,
  if $F\in [\I_S, \A]_{add}$ then $F\circ f^*\in [\I_T,\A]_{add}$.

  We must show that for disjoint $A,B\subseteq T$ there exists an isomorphism
  $\psi_{A,B}$ making the following diagram commute:
  \begin{diagram}
    { F(f^{-1}(A)) \\
      F(f^{-1}(A)) \oplus F(f^{-1}(B)) & F(f^{-1}(A \cup B)) \\
      F(f^{-1}(B)) \\};
    \to{1-1}{2-2} \to{1-1}{2-1} \to{3-1}{2-1} \to{2-1}{2-2}^{\psi_{A,B}} \to{3-1}{2-2}
  \end{diagram}
  Since $A$ and $B$ are disjoint, $f^{-1}(A)$ and $f^{-1}(B)$ are also disjoint.
  Since $F$ is an addition functor, setting
  $\psi_{A,B} = \phi_{f^{-1}(A),f^{-1}(B)}$ works, as desired.

  It remains to check that $D_\A$ is special. Since $D_\A(*) = *$, it suffices
  to check that the functor
  \[\rho_n: [\I_{\mathbf{n}},\A]_{add} \rto [\I_{\mathbf{1}},\A]_{add}^n\]
  is an equivalence of categories.  A functor $\I_{\mathbf{1}} \rto \A$ is
  simply a choice of object in $\A$.  Thus the fact that $\rho_n$ is an
  equivalence of categories is simply the statement that an addition functor
  $F:\I_{\mathbf{n}} \rto \A$ is uniquely determined (up to natural isomorphism)
  by its values on $\{*,1\},\ldots,\{*,n\}$, which is exactly what the
  definition of an addition functor ensures. 
\end{proof}

The following is a consequence of the definitions:
\begin{proposition} \label{prop:DtoU}
  There is a morphism of multisimplicial categories
  $u_n: D_\A(S^n) \rto iU_\dotp^{(n)}(\A)$ which is a levelwise equivalence of
  categories.  Consequently, these assemble into a morphism of symmetric spectra
  $u: K(D_\A) \rto U(\A)$ which is a level equivalence.
\end{proposition}

\begin{proof}
  The simplices of $iU_\dotp^{(n)}(\A)$ contain choices of sums for all
  contiguous subsets of an ordered grid.  The simplices of $D_\A(S^n)$ contain
  choices of sums for all subsets of the same ordered grid.  Thus $u_n$ simply
  forgets the choices. (NB: it is important for the existence of this grid that
  we are using the definition $S^n = (S^1)^{\wedge n}$, with the standard
  definition of $S^1$.)  That it is an equivalence of categories follows from
  (CA2).  The rest of the proposition follows directly from the fact that
  equivalences of categories induce homotopy equivalences on classifying spaces.
\end{proof}

\subsection{Valuations on convenient addition categories}

\begin{definition}
  Let $\A$ be a convenient addition category, and $\P$ a poset.  A
  \emph{$\P$-valuation} on $\A$ is a functor $\A^{mon} \rto \P$, where
  $\A^{mon}$ is the subcategory of monomorphisms of $\A$.
\end{definition}

\begin{proposition}
  A $\P$-valuation on a convenient addition category $\A$ induces a
  $\P$-valuation on $D_\A$ and on $U(\A)$, which agree, in the sense that $u$
  induces an equivalence $F_pK(D_\A) \rto F_pU(\A)$ for all $p\in \P$.
\end{proposition}

\begin{proof}
  Let $\nu: \A \rto \P$ be the $\P$-valuation on $\A$.  We define a
  $\P$-valuation $\nu_D$ on $D_\A$ by setting $\nu_D(F) = \nu(F(S))$ for a
  functor $F \in[ \I_S,\A]_{add}$.  Given a morphism $f: S \rto T$ in
  $\FinSet_*$, it induces a morphism $F\circ f^* \rto F$ in $\El D_\A$; since
  $f^{-1}(T^\circ) \subseteq S^\circ$; thus $f$ produces a morphism
  $\nu(F(f^{-1}(T))) \rto \nu(F(S))$, as desired.  (Since $\P$ is a poset we do
  not need to check that composition is respected to check that $\nu$ is a
  functor.)  We define a $\P$-valuation $\nu_U$ on $U_\dotp^{(n)}(\A)$ by
  setting $\nu_U(X) = \nu(\max X)$ for a simplex $X\in U_\dotp^{(n)}(\A)$.  In
  particular, this implies that $\nu_U(u_n(x)) = \nu_D(x)$ for any
  $x\in D_\A(S^n)$, so $u_n$ (and therefore $u$) are compatible with the
  filtration.  Since $u_n$ is a filtered equivalence of multisimplicial
  categories, it induces an equivalence of filtered components as well.
\end{proof}

\subsection{A brief remark about Waldhausen categories}

We can extend the definition of a $\P$-valuation on a convenient addition
category to one on a Waldhausen category $\C$ by defining a $\P$-valuation to be
a functor $c\C \rto \P$, where $c\C$ is the subcategory of cofibrations of $\C$.
Via a similar construction to the one on $U_\dotp$, the valuation extends to a
valuation on $S_\dotp$ and thus to $K(\C)$.  This is exactly the way that Rognes
defines his filtration in \cite{rognes}.  We will discuss the implications of
this in in Section~\ref{sec:rogneswald} but will not discuss this in detail, as
Rognes has already discussed it at length in the paper, and all of the
constructions in this paper are directly inspired by those in Rognes's paper.

The case of Waldhausen categories is interesting, in addition, because for a
Waldhausen category $\C$, the category $c\C$ of cofibrations is often a
convenient addition category (with the monoidal structure given by the
coproduct).  In these cases, $U_\dotp(c\C)$ is exactly a model for the the
direct sum $K$-theory of $\C$.  Even in those cases where it is not a convenient
addition category, there exists a comparison functor to $K(\C)$ by taking a grid
of objects to its ``partial sums.''

\section{Subobject structures in convenient addition categories} \label{sec:subobjects}

\subsection{Computing filtered components in terms of subobject structures}

We work in a choice of convenient addition category $\A$.

\begin{definition}
  Let $A\in \A$.  Write $\A_{/A}^{mon}$ for the full subcategory of $\A_{/A}$ of
  those morphisms $f: B \rto A$ which are monomorphisms.  Note that
  $\A_{/A}^{mon}$ is a preorder.

  A \emph{subobject structure} on $A$, denoted $\Sub_A$ is a choice of skeleton
  of $\A_{/A}^{mon}$.  This choice is unique up to unique isomorphism, and we
  will by an abuse of language sometimes refer to ``the subobject structure'' on
  $A$.

  We will often write ``let $B\in \Sub_A$'', suppressing the morphism component;
  when we want to refer to the morphism $B \rcofib A$ we will call it $\iota_B$.

  Any object $B \in \Sub_A$ induces a subobject structure on $B$ by considering
  the full subcategory of objects above $B$ (and forgetting the monomorphism to
  $A$).  
\end{definition}

\begin{definition}  \label{def:autaction}
  The category $\Sub_A$ comes with an action by the group $\Aut(A)$.  The
  element $g\in \Aut(A)$ takes an object $B\in \Sub_A$ to the object
  $B'\in \Sub_A$ for which there exists an isomorphism $B \rto B'$ making the
  diagram
  \begin{diagram}
    { B & B' \\
      A & A \\};
    \to{1-1}{1-2}^\cong \to{2-1}{2-2}^g
    \cofib{1-1}{2-1}_{\iota_B} \cofib{1-2}{2-2}^{\iota_{B'}}
  \end{diagram}
  commute.  Since $B'$ is unique, this action is well-defined.
\end{definition}

\begin{definition}
  Let $B,B'\in \Sub_A$.  If there exists a monomorphism
  $\iota: B \oplus B' \rcofib A$ in $\A$ making the diagram
  \begin{diagram}
    { B & B \oplus B' & B' \\
      & A \\};
    \cofib{1-1}{1-2} \cofib{1-3}{1-2} \cofib{1-2}{2-2}^\iota
    \cofib{1-1}{2-2} \cofib{1-3}{2-2}
  \end{diagram}
  commute, then there exists a unique $C\in \Sub_A$ and isomorphism
  $B \oplus B' \rto^\cong C$ in $\A/A$.  We write $B\relplus B'\defeq C$, as
  objects of $\Sub_A$.  
\end{definition}

Note that $\relplus$ is a partially defined symmetric monoidal structure: it is
unital, $(B\relplus B') \relplus B''$ is defined if and only if
$B \relplus (B' \relplus B'')$ is defined, and the associator exists whenever it
is defined (and similarly for the commutator).  Moreover, $\relplus$ is in fact
\emph{strictly} associative, commutative, and unital whenever it is defined.
The subcategory of $\Sub_A\times \Sub_A$ on which $\relplus$ is defined is a
sieve.

\begin{definition}
  Let $\Sub_A$ be a subobject structure on $A$.  Define the $\Gamma$-set
  $S_{\leq A}$ to be the $\Gamma$-subset of $\ob D_\A$ of those functors which factor
  through the forgetful functor $U:\Sub_A \rto \A$.  In other words, an object
  of $S_{\leq A}(S)$ can be encoded as a tuple $\{A_s\}_{s\in S^\circ}$ of
  objects in $\Sub_A$, such that for any $T \subseteq S$, $\bigrelplus_{t\in T}
  A_t$ is defined in $\Sub_A$.  In this encoding, a function
  $f:S \rto T$ is taken to the function which maps a tuple
  $\{A_s\}_{s\in S}$ to the tuple $\{B_t\}_{t\in T}$ where
  \[B_t \defeq \bigrelplus_{s \in f^{-1}(t)} A_s.\]
  The $\Gamma$-set
  $S_{<A}$ is defined to be the $\Gamma$-subset of $S_A$ of those functors that
  factor through $\Sub_A$ without hitting $A$.  Also define
  \[S_A \defeq S_{\leq A}/S_A.\]
\end{definition}

The key point of all of these definitions is the following results, which
describe the filtered part of the $K$-theory of a convenient addition category
in terms of a subobject structure.

The following is direct from the definitions:
\begin{lemma} \label{lem:filt-to-gamma}
  \[F_p K(D_\A) = K(F_p D_\A),\]
  levelwise.
\end{lemma}

Note that although $F_pD_\A$ is a $\Gamma$-set, it is not a special
$\Gamma$-set.

The $\Aut(A)$-action on $\Sub_A$ (see Definition~\ref{def:autaction}) extends to
$\Aut(A)$-actions on $S_{\leq A}$, $S_{<A}$ and $S_A$ by postcomposition.  In
many cases, this action models the filtered portions of the $\P$-valuation:

\begin{theorem} \label{thm:graded-part-SA}
  Let $\A$ be a convenient addition category equipped with a $\P$-valuation.
  Suppose in addition that if $\nu(A) = \nu(B)$ then $A \cong B$.  Then for all
  $p\in \P$,
  \[F_pK(\A)_n/F_{<p}K(\A)_n \simeq (K(S_A)_n)_{h\Aut(A)} \qqand F_pK(\A)/F_{<p}
    K(\A) \simeq K(S_A)_{h\Aut(A)},\] where $A$ is and object of $\A$ such that
  $\nu(A) = p$.
\end{theorem}

\begin{proof}
  By the condition on $\nu$, the right-hand side is well-defined, since the
  choices of $A$ are all isomorphic, and thus their subobject structures and
  groups of automorphisms are isomorphic.  Let $\Sub_A$ be a chosen subobject
  structure on $A$.

  The second part of the theorem is a direct consequence of the first, so we
  focus on proving the first statement.

  By Lemma~\ref{lem:filt-to-gamma}
  and a direct comparison of the simplicial structure,
  \[F_p K(\A)/F_{<p} K(\A) \simeq K(F_p D_\A)/K(F_{<p} D_\A) \simeq K(F_p
    D_\A/F_{<p}D_\A).\] Thus to prove the theorem it suffices to show that
  \[K(F_pD_\A/F_{<p} D_\A)_n \simeq (K(S_A)_{h \Aut(A)})_n \simeq (K(S_A)_{h
      \Aut(A)})_n.\]
  To check this it suffices to prove a much more direct statement:
  \[(NF_pD_\A(S))/(NF_{<p}D_\A(S)) \simeq S_A(S)_{h\Aut(A)}\] for any finite
  pointed set $S$.  The space on the right is modeled as $N\D$, where $\D$ is
  the category with
  \begin{description}
  \item[objects] elements of $S_A(S)$ (plus a basepoint), and
  \item[morphisms] $\Hom(x,y) = \{g\in \Aut(A) \,|\, g\cdot x = y.\}$ (with no
    morphisms into or out of the basepoint),
  \end{description}
  with composition via multiplication in $\Aut(A)$.  The functor
  $\D \rto F_pD_\A/F_{<p} D_\A$ is induced by the forgetful functor on objects.
  any element in $\Aut(A)$ induces a natural action on $\Sub_A$, and thus by
  postcomposition acts on the objects of $\D$.  Moreover, it is essentially
  surjective on the subcategory of $F_p D_\A(S)$ of those functors whose value
  on $S^\circ$ is isomorphic to $A$, and thus it is essentially surjective when
  mapped to the quotient $F_p D_\A(S)/F_{<p} D_\A(S)$.  (Here we are using the
  fact that if $\nu(A) = \nu(B)$ then $A \cong B$.) Thus to prove that it is an
  equivalence it suffices to prove that it is full and faithful.  Faithfulness
  is immediate: two different elements in $\Aut(A)$ map to different
  isomorphisms, since when restricted to the $S$-coordinate of the functor
  $\I_S \rto \A$ they will be different.  For fullness, it suffices to check
  that for any $T \subseteq S^\circ$, given $X,Y: \I_S \rto \Sub_A$, the
  function
  \[\operatorname{NatIso}(X,Y) \rto \Aut(A) \qquad \alpha \rgoesto \alpha_S\]
  is injective.  (This shows that any natural isomorphism is uniquely determined
  by its value at $S$, and since the morphisms in $\D$ are defined to be
  precisely those induced correctly on the $S$-coordinate, this shows that the
  functor is full.)  In any diagram
  \begin{diagram}
    { X(T) & X(S^\circ) \\ Y(T) & Y(S^\circ) \\};
    \arrowsquare{}{\alpha_T}{\alpha_S}{}
  \end{diagram}
  the horizontal morphisms are monic, which implies that for any choice of
  $\alpha_S$, if $\alpha_T$ exists to complete the diagram, it must be unique.
  This shows that the above function is injective, as desired.
\end{proof}

We have therefore shown that the filtration components of $K(\A)$ can be
investigated by understanding the $K$-theory of $S_A$ and the action of
$\Aut(A)$ on this $K$-theory.

\begin{example}
  Consider the case when $\A = \FinSet^{inj}$ and $\P = \Z_{\geq 0}$.  We can
  define a $\P$-valuation by sending a finite set to its cardinality.  Then
  $S_n$ is the $\Gamma$-set of those functors whose value on $S$ is
  $\mathbf{n}$, and $\Aut(\mathbf{n}) = \Sigma_n$.  Thus
  Theorem~\ref{thm:graded-part-SA} states that
  \[F_n/F_{n-1}K(\FinSet^{inj}) \simeq K(S_{n})_{h\Sigma_n}.\]
\end{example}

\begin{example}
  Consider the case when $\A=\Vect_k$ and $\P = \Z_{\geq 0}$.  We define a
  $\P$-valuation by sending a vector space to its dimension.  Then $S_{k^n}$ is
  the $\Gamma$-set of those fucntors whose value on $S$ is $k^n$, and
  $\Aut(k^n) = \GL_n(k)$.  Thus Theorem~\ref{thm:graded-part-SA} states that
  \[F_n/F_{n-1}K(\Vect_k) \simeq K(S_{k^n})_{hGL_n(k)}.\]
  (This is analogous to Rognes's \cite[Proposition 3.8]{rognes}.)
\end{example}

\begin{example}
  Consider the case when $k$ is a square-root closed ordered field, $\A = \Inn_k$, and $\P =
  \Z_{\geq 0}$.  We can similarly define a $\P$-valuation by taking a space to
  its dimension.
  Applying Theorem~\ref{thm:graded-part-SA} analogously to the
  above examples gives
  \[F_n/F_{n-1}K(\Inn_k) \simeq K(S_{k^n})_{h O_n(k)}.\]
\end{example}


\subsection{Analyzing the $K$-theory of subobject structures}

\begin{definition} \label{def:dec}
  Let $A\in \A$.  The \emph{decomposition category of $A$}, written $\Decomp_A$,
  is defined as follows.
  \begin{description}
  \item[objects] Objects are finite sets $\{A_1,\ldots,A_n\}$, of non-initial
    objects in $\Sub_A$ such that $\relplus_{i=1}^n A_i = A$.
  \item[morphisms] A morphism $\{B_1,\ldots,B_m\} \rto \{A_1,\ldots,A_n\}$ is a
    function $f:\{1,\ldots,n\} \rto \{1,\ldots,m\}$ such that for all
    $1 \leq j \leq m$, $\relplus_{i\in f^{-1}(j)} A_i = B_j$.
  \end{description}
  For an object $p = \{A_1,\ldots,A_n\}$ in $\Decomp_A$, we write $|p| \defeq
  n$.
  
  Let $\Decomp_A^\circ$ be the full subcategory of $\Decomp_A$ except for
  the initial object $\{A\}$---i.e. it is the full subcategory of those $p$ with
  $|p| \geq 2$.
\end{definition}

Note that $\Decomp_A$ is closed under finite products and has an initial object
$\{A\}$.

\begin{lemma}
  If there is a morphism $\{B_1,\ldots,B_m\} \rto \{A_1,\ldots,A_n\}$ in $\Decomp_A$
  then for each $1 \leq i \leq n$ there is a unique $1 \leq j \leq m$ such that
  there exists a monomorphism $A_i \rto B_j$.  Therefore $\Decomp_A$ is a
  poset.
\end{lemma}

\begin{proof}
  Consider the following special case.  Suppose that $Y \relplus Z$ is defined
  and that $X$ is such that $X \rto Y$ and $X \rto Z$ are morphisms in
  $\Sub_A$.  We claim that this implies that $X \cong \initial$.
  In particular, this implies that the diagram
  \begin{diagram}
    { X & Y \\ Z & Y\oplus Z \\
      & &  A \\};
    \arrowsquare{}{}{}{}
    \to{1-2}{3-3} \to{2-1}{3-3} \to{2-2}{3-3}
  \end{diagram}
  commutes.  Since $Y \times_{Y\oplus Z} Z =\initial$ by (CA3), this implies
  that there exists a morphism $X \rto \initial$; by (CA1) this is an
  isomorphism.

  A morphism $(B_1,\ldots,B_m) \rto (A_1,\ldots,A_n)$ is determined by a
  function $f: \{1,\ldots,n\} \rto \{1,\ldots,m\}$.  Given such a morphism,
  there exists a morphism $A_i \rto B_{f(i)}$ for all $1 \leq i \leq n$.  Thus
  to prove the lemma it suffices to check that the function is unique.  Thus
  suppose that there are two functions $f,f'$ which satisfy the condition in the
  definition, and suppose that $f(i) \neq f'(i)$ for some $i\in \{1,\ldots,n\}$.
  This means that there exist morphisms $A_i \rto B_{f(i)}$ and
  $A_i \rto B_{f'(i)}$ in $\Sub_A$.  By the above observation, this implies that
  $A_i \cong \initial$, contradicting the definition of $\Decomp_A$.  Thus $f$
  is unique, as desired.
\end{proof}

There is a $\Decomp_A$-valuation on $S_A$ given by sending any element
$\{A_1,\ldots,A_i\}$ to the set
\[\{A_j \,|\, j\in \mathbf{i},\ A_j \neq I\}.\] We call this the \emph{canonical
  decomposition valuation} on $S_A$ and denote by $\delta$.

The filtered stages of the $\Gamma$-set $S_A$ turns out out to be surprisingly
simple to analyze.

\begin{definition}
  For any pointed simplicial set $L$ and pointed finite set $T$ write
  $L^{\wedge T}$ for the pointed simplicial set whose $n$-simplices are given by
  the set of pointed functions $T \rto L_n$ which are either constant at the
  basepoint or injective.  The face and degeneracy maps are constructed by
  postcomposition.
\end{definition}

\begin{proposition} \label{prop:level-analysis-and-decompositions}
  Let $L: \Delta^\op \rto \FinSet_*$ be a simplicial pointed finite set, and let
  $\A$ be a convenient addition category and $A\in \A$.  Then for all $p\in
  \Decomp_A$, $F_pS_A(L) \cong L^{\wedge p}$.  
  Moreover, there is a $(\Decomp_A)_{\leq p}$-valuation on $L^{\wedge p}$ making
  this isomorphism a filtered isomorphism.
\end{proposition}

\begin{proof}
  A non-basepointe $n$-simplex in $S_A(L)$ with valuation $q$ is given by a
  function $f:L_n^\circ \rto \ob\Sub_A$ such that
  $f(L_n^\circ) \cap (\ob\Sub_A \backslash \{\initial\}) = q$.  Moreover, by the
  definition of $S_A$, $f$ is injective away from $\{\initial\}$, so such a
  simplex is uniquely defined by a function $f': q \rto L_n^\circ$.  Given a
  morphism $q \rto p$ in $\Decomp_A$ $f'$ uniquely defines (and is defined by) a
  function $\widehat f: p \rto L_n^\circ$.  Thus the $n$-simplices of
  $F_pS_A(L)$ is exactly the set $(L_n^\circ)^{|p|}$.  The face and degeneracy
  maps are given by postcomposing this map with the appropriate face and
  degeneracy maps, so we see that $S_A(L) \cong L^{\wedge |p|}$, as desired.

  The valuation on $L^{\wedge p}$ takes a function $f: p \rto L_n$ to the minimum
  $q \rto q$ such that there exists a function $f': q \rto L_n$ making the
  diagram
  \begin{diagram}
    { p & L_n \\
      q \\};
    \to{1-1}{1-2}^f \to{2-1}{1-1} \to{2-1}{1-2}_{f'}
  \end{diagram}
  commute.
\end{proof}

\begin{corollary} \label{cor:KSA}
  $S_A$ satisfies the conditions of Corollary~\ref{cor:gamma-space-indec}, and
  thus
  \begin{align*}
    K(F_pS_A) &\simeq \Sigma^\infty \Sigma' N (\Decomp_A^\circ)_{\leq p} \\
    K(F_{<p}S_A) &\simeq \Sigma^\infty \Sigma' N(\Decomp_A^\circ)_{<p} \\
    K(S_A) &\simeq \Sigma^\infty \Sigma' N (\Decomp_A^\circ)
  \end{align*}
\end{corollary}




\section{Applications} \label{sec:applications}

This section contains some applications and consequences of the theory developed in the previous sections.  It also relies heavily on a result of Alexander Kupers in Appendix~\ref{app:kupers} which shows that, in many cases, the spaces $N \Decomp_A^\circ$ are homotopy equivalent to a wedge of spheres.

\subsection{Finite Sets}
  Consider the case $\A = \FinSet$, with $\oplus = \amalg$.  Let $A =
  \mathbf{n}$.  Then $\Decomp_A$ is the category of unordered partitions of
  $\A$, ordered by refinement.  $\Decomp^\circ_A$ is then the full
  subcategory of nontrivial partitions.  When $n > 1$ this has a terminal
  object (given by the discrete decomposition) and is therefore contractible.
  When $n=1$, $\Decomp^\circ_A$ is empty, and thus $K(S_A) \simeq \S$.

  Let $d\in \Decomp_A$ be the discrete partition (i.e. the terminal object),
  with $|d|=n$. Then by the valuation constructed in the proof of
  Proposition~\ref{prop:level-analysis-and-decompositions}, $F_{<d}S_A(L)_m$ is
  the set of functions $p \rto L_m$ such that at least two of the coordinates
  are the same---in other words, it is the fat diagonal in $L^{\wedge p}$.  In
  the case when $L = S^k$, this is the fat diagonal in $S^{kn}$.  By Alexander
  duality,
  \[\tilde H_q(F_{<d}S_A(S^k)) \cong \tilde
    H^{nk-q-1}(\mathrm{Conf}_n(\R^k)).\] By \cite[Proposition
  3.1.4]{knudsen18} (originally in \cite{cohen73}), 
  $H^*(\mathrm{Conf}_n(\R^k))$ is generated as an algebra in degree $k-1$, and
  is concentrated in degrees $0,k-1,\ldots,(k-1)(n-1)$, with all groups free.
  Thus we see that $\tilde H_q(F_{<d} S^{nk})$ is free and concentrated in
  degrees $nk-(n-1)(k-1)-1, nk-(n-1)(k-2)-1, \ldots, nk-(n-1)-1$.  The lowest
  degree is therefore $n+k-2$.  Stabilizing with respect to $k$, the homology of
  $F_{<\delta_k} K(S_A)$ will be concentrated in degree $n-2$.

  For any $p\in \Decomp_A$, $(\Decomp_A)_{\leq p} \cong \Decomp_{p}$, and thus
  $(\Decomp_A)_{<p} \cong \Decomp_p \backslash \{\mathrm{discrete}\}$.  Thus the
  above case is sufficient to understand any of the filtered portions of
  $\Decomp_A$. 

\begin{corollary}[Barratt--Priddy--Quillen Theorem, \cite{barratt_priddy}]
  \[K(\FinSet_*) \simeq \S\]
\end{corollary}

\begin{proof}
  This is the same proof as found in \cite[Section 4]{rognes}, but stated in
  terms of the language in this paper.  We have a $\Z_{\geq 0}$-valuation on
  $\FinSet_*$ given by the size of the set.  By Corollary~\ref{cor:KSA} we have
  \[K(F_d S_{\mathbf{n}}) \simeq \Sigma^\infty \Sigma'' N (
    \Decomp^\circ_{\mathbf{n}}) \simeq
    \begin{cases}
      \S \caseif n = 1\\
      * \caseotherwise.
    \end{cases}\]
  By Theorem~\ref{thm:graded-part-SA}, for $n > 1$ we have
  \[F_n K(\A) / F_{n-1} K(\A) \simeq K(S_{\mathbf{n}})_{h\Sigma_n} \simeq *;\]
  thus \[K(\A) \simeq F_1(\A) \simeq \Sigma^\infty S^0 = \S,\]
  by Corollary~\ref{cor:KSA}.
\end{proof}


\subsection{Inner product spaces}

Let $k$ be a ordered field.  By Theorem~\ref{thm:graded-part-SA}, the dimension
valuation on $K(\Inn_k)$ has
\[F_n/F_{n-1} K(\Inn_k) \simeq K(S_{k^n})_{h O_n(k)}.\] By
Corollary~\ref{cor:Gsusp} there is an $O_n(k)$-equivariant equivalence
\[K(S_{k^n}) \simeq \Sigma^\infty \Sigma' N(\Decomp^\circ_{k^n}).\]

\begin{lemma} \label{lem:innsphere}
  $N\Decomp^\circ_{k^n}$ is a wedge of $n-2$-spheres.  
\end{lemma}

\begin{proof}
This is an application of Theorem~\ref{thm:deorder}.  By Theorem~\ref{thm:deorder}, $\Decomp^\circ_{k^n}$ is $n-2$-connected if and only if $(\Decomp_{k^n}^{ord})^\circ$ is.  But $N(\Decomp^{ord}_{k^n})^\circ$ is isomorphic to the barycentric subdivision of the Tits building of $k^n$ (where a flag is associated to a decomposition by taking $U_0 \subseteq \cdots \subseteq U_m$ to $(U_0,U_0^\perp \cap U_1,\ldots,U_{m-1}^\perp\cap U_m)$.  Since the Tits building is homotopy equivalent to a wedge of $n-2$-spheres \cite[Theorem 2]{quillen73-fingen}, so is $N(\Decomp_{k^n}^{ord})^\circ$. Thus $N(\Decomp^\circ_{k^n})$ is a wedge of $n-2$-spheres, as desired.
\end{proof}

\begin{definition}
    Let
    \[DO_n \defeq \tilde H_{n-2}(N \Decomp^\circ_{k^n};\Q). \]
\end{definition}

\begin{corollary}
\[\pi_i\left(K(S_{k^n})_{hO(n)}\right) \otimes \Q \cong H_i\left(K(S_{k^n})_{hO(n)}; \Q\right) \cong H_{i-n+1}\left(O(n), DO_n\right).\]
\end{corollary}

\begin{proof}
We prove this by showing that the same spectral sequence converges to both.  

The homotopy orbit spectral sequence states that there is a spectral sequence
\[E^2_{s,t} = H_s(O(n), \pi_t K(S_{k^n})) \Rto \pi_{s+t}K(S_{k^n})_{hO(n)}.\]
Thus to prove the claim it suffices to show that 
\[\pi_t K(S_{k^n}) \cong \begin{cases}
    \Q \caseif n=1 \hbox{ and } t = 0, \\
    DO_n \caseif t = n-1, \\
    0 \caseotherwise.
    \end{cases}\]
By Corollary~\ref{cor:Gsusp}, there is an $O_n(k)$-equivariant equivalence $K(S_{k^n}) \simeq \Sigma^\infty \Sigma' N\Decomp^\circ_{k^n}$.    By Lemma~\ref{lem:innsphere}, $N\Decomp^\circ_{k^n}$ is homotopy equivalent to a wedge of $n-2$-spheres, so the spectrum $\Sigma^\infty \Sigma' N\Decomp^\circ_{k^n}$ has rational homotopy groups which are $0$ everywhere except in degree $n-1$, where it is a rational vector space isomorphic to  $DO_n$, as desired.  (We need to be a bit careful for the case $n=1$. Here $\Decomp^\circ_{k^1}$ is empty, so the spectrum $\Sigma^\infty \Sigma' N \Decomp^\circ_{k^n}$ is the sphere spectrum, giving the desired special case.)

The same argument shows that there is a spectral sequence
\[E^2_{s,t} = H_s(O(n), H_t (K(S_{k^n});\Q) \Rto H_{s+t}((K(S_{k^n})_{hO(n)};\Q),\]
since all that matters for the proof of the spectral sequence is that a homology theory is being applied (and both rational stable homotopy and rational homology are homology theories).  But because $N\Decomp^\circ_{k^n}$ is a wedge of sphere spectra, its rational stable homotopy and its homology are isomorphic.  Thus this spectral sequence is isomorphic at the $E^2$-page to the previous one. Since it collapses at the $E^2$ page the groups are isomorphic, as desired.
\end{proof}

In particular, this implies that the stable rank filtration  on $K(\Inn_k)$ produces a spectral sequence converging to $K_{s+t}(\Inn_k)\otimes\Q$ with $E^1$-page isomorphic to
\begin{center}
\begin{diagram}
  {3 & 0 & H_3(O(2); DO_2) & H_3(O(3); DO_3) & \cdots \\
   2 & 0 & H_2(O(2); DO_2) & H_2(O(3); DO_3) & \cdots \\
   1 & 0 & H_1(O(2); DO_2) & H_1(O(3); DO_3) & \cdots \\
   0 & \Q & H_0(O(2);DO_2)  & H_0(O(3); DO_3) & \cdots \\
   \phantom{*}  & 0                       & 1                       & 2
  \\};
  \node (A) at (m-4-2) {\phantom{$H_0(GL1_1)$}};
  \draw[->] (A.south west) -- +(0,5);
  \draw[->] (A.south west) -- +(10,0); 
  \to{1-3}{1-2} \to{1-4}{1-3}
  \to{2-4}{2-3} \to{2-3}{2-2}
  \to{3-3}{3-2} \to{3-4}{3-3}
  \to{1-5}{1-4} \to{2-5}{2-4} \to{3-5}{3-4}
  \to{4-5}{4-4} \to{4-4}{4-3} \to{4-3}{4-2}
\end{diagram}
\end{center}
Here the entries in the first column are $H_t(O(1),DO_1)$, but $DO_1 \cong \Q$
and $O(1) \cong \Z/2$, so these will all be zero after the $t=0$ case.  The
spectral sequence is all $0$ below the line $s=t$.  It is important to keep in
mind that all groups in this spectral sequences are \textbf{discrete}, because
the category $\Inn_k$ is not a topological category, and the groups $O(n)$ are
arising as automorphism groups of objects in $\Inn_k$.

\subsection{Rings and fields}

Let $k$ be a field, or a Dedekind domain.  Let $\Free_k$ be the category of finitely generated free $k$-modules and monomorphisms.  This is a convenient addition category with ranks (given by the rank of a module). By Theorem~\ref{thm:graded-part-SA}, the dimension
valuation on $K(\Free_k)$ has
\[F_n/F_{n-1} K(\Free_k) \simeq K(S_{k^n})_{h GL_n(k)}.\] By
Corollary~\ref{cor:KSA}
\[K(S_{k^n}) \simeq \Sigma^\infty \Sigma' N(\Decomp^\circ_{k^n}).\]
 
\begin{proposition} 
$N\Decomp^\circ_{k^n}$ is homotopy equivalent to a wedge of $n-2$-spheres.
\end{proposition}

For finite fields, this is a theorem of Welker; see \cite{welker95}.

\begin{proof}
This is an application of Theorem~\ref{thm:deorder}.  By Theorem~\ref{thm:deorder}, $\Decomp^\circ_{k^n}$ is homotopy equivalent to a wedge of $n-2$-spheres if and only if $(\Decomp_{k^n}^{ord})^\circ$ is.  But $(N\Decomp_{k^n}^{ord})^\circ$ is isomorphic to the barycentric subdivision of the split Tits building of $k^n$ (also referred to as the \emph{$E_1$-splitting complex} in \cite[Definition 17.9]{galatius_kupers_randal-williams18}) where a simplex $(P_0,Q_0) \subseteq \cdots \subseteq (P_n,Q_n)$ of the split Tits building is associated to the decomposition $(P_0,P_1\cap Q_0,\ldots,P_{n-1}\cap Q_{n-2},Q_n)$.  Since the split Tits building is homotopy equivalent to a wedge of $n-2$-spheres \cite[Theorem 1.1]{charney80}, so is $N(\Decomp_{k^n}^{ord})^\circ$. Thus $N(\Decomp^\circ_{k^n})$ is a wedge of $n-2$-spheres, as desired.
\end{proof}

\begin{definition}
    Let 
    \[\Dec_n \defeq \tilde H_{n-2}(N \Decomp^\circ_{k^n};\Q). \]
\end{definition}

By an analogous argument to the previous part, we see that
\[\pi_i\left(K(S_{k^n})_{hGL_n(k)}\right) \otimes \Q  \cong H_i\left(
    K(S_{k^n})_{hGL_n(k)}; \Q\right) \cong H_{i-n+1}(GL_n(k), \Dec_n).\]
This gives a spectral sequence
\begin{center}
\begin{diagram}
  {3 & H_3(k^\times;\Q) & H_3(GL_2(k); \Dec_2) & H_3(GL_3(k); \Dec_3) & \cdots \\
   2 & H_2(k^\times;\Q) & H_2(GL_2(k); \Dec_2) & H_2(GL_3(k); \Dec_3) & \cdots \\
   1 & H_1(k^\times;\Q) & H_1(GL_3(k); \Dec_2) & H_1(GL_3(k); \Dec_3) & \cdots \\
   0 & \Q & H_0(GL_k(k);\Dec_2)  & H_0(GL_3(k); \Dec_3) & \cdots \\
   \phantom{*}  & 0                       & 1                       & 2
  \\};
  \node (A) at (m-4-2) {\phantom{$H_0(GL1_1)$}};
  \draw[->] (A.south west) -- +(0,5);
  \draw[->] (A.south west) -- +(10,0); 
  \to{1-3}{1-2} \to{1-4}{1-3}
  \to{2-4}{2-3} \to{2-3}{2-2}
  \to{3-3}{3-2} \to{3-4}{3-3}
  \to{1-5}{1-4} \to{2-5}{2-4} \to{3-5}{3-4}
  \to{4-5}{4-4} \to{4-4}{4-3} \to{4-3}{4-2}
\end{diagram}
\end{center}
converging to $K_{p+q}(k)\otimes \Q$ (above $p+q=0$, where it may be different
because we are looking at $K(\Free_k)$ and not $K(\mathbf{Proj}_k)$
\cite[Theorem 4.11(b)]{weibel_kbook}).  (See also \cite[Remark 13.30]{galatius_kupers_randal-williams18}, where a closely-related spectral sequence is discussed.)

If $k$ is an infinite field,
\cite[Theorem 8.1]{galatius_kupers_randal-williams} states that when $q < p$ this is $0$. This simplifies
this spectral sequence to the following form:
\begin{center}
\begin{diagram}
  {3 & H_3(GL_1) & H_3(GL_2, \tilde H_0(\Dec_2)) & H_3(GL_3, \tilde H_1(\Dec_3)) \\
   2 & H_2(GL_1) & H_2(GL_2, \tilde H_0(\Dec_2)) & H_2(GL_3, \tilde H_1(\Dec_3)) \\
   1 & H_1(GL_1) & H_1(GL_2, \tilde H_0(\Dec_2)) & 0 \\
   0 & \Z & 0 & 0 \\
   \phantom{*}  & 0                       & 1                       & 2
  \\};
  \node (A) at (m-4-2) {\phantom{$H_0(GL1_1)$}};
  \draw[->] (A.south west) -- +(0,5);
  \draw[->] (A.south west) -- +(10,0); 
  \draw[color=green!70!black] (m-3-3.south west) rectangle (m-3-3.north east);
  \draw[color=green!70!black] (m-2-4.south west) rectangle (m-2-4.north east);
  \to{1-3}{1-2} \to{1-4}{1-3}
  \to{2-4}{2-3} \to{2-3}{2-2}
  \to{3-3}{3-2}
\end{diagram}
\end{center}
The groups in green boxes are torsion if $k$ is an infinite field.  In this case, this gives an alternate form of the spectral sequence discussed by Rognes in \cite[Section 12]{rognes21}, which could be compared to the chain complexes desired by Beilinson and Lichtenbaum.  

\begin{question}
    How do the rows in this spectral sequence compare to the rows in Rognes's?  How do they compare to motivic complexes?
\end{question}

\subsection{Rognes's poset filtration} \label{sec:rogneswald}

Consider Rognes's spectrum $D(R^k)$ \cite[Definition 3.9]{rognes} for a commutative ring $R$.  We can
rephrase his poset filtration in terms of a valuation in the following way.

\begin{definition} \label{def:spanning}
  Let $R$ be a commutative ring.  Consider the convenient addition category of free $R$-modules and monomorphisms with free cokernel;  let $\Sub_{R^k}$ be a submodule structure
  on $R^k$.

  A \emph{spanning poset} in $\Sub_{R^k}$ is a full subposet $\mathcal{S}$ of
  $\Sub_{R^k}$ such that the following conditions hold:
  \begin{itemize}
  \item[(SP1)] $\mathcal{S}$ contains $0$ and is closed under intersections in
    $\Sub_{R^k}$ 
  \item[(SP2)] The colimit of $\mathcal{S}$ in $\Sub_{R^k}$ is $R^k$. 
  \end{itemize}
  A spanning poset is \emph{minimal} if it contains no proper subposet which is
  also a spanning poset; it is \emph{trivial} if it is the poset $0 \rto R^k$.  A spanning poset $\mathcal{S}$ is equipped with a
  canonical inclusion $\iota_{\mathcal{S}}: \mathcal{S} \rto \Sub_{R^k}$.

  Let $\P$ be the poset with
  \begin{description}
  \item[objects] minimal spanning posets in $\Sub_{R^k}$,
  \item[relation] $\mathcal{S} \leq \mathcal{S}'$ if there exists a functor $F:
    \mathcal{S}' \rto \mathcal{S}$ such that there is a natural transformation
    $\iota_{\mathcal{S}'} \Rto \iota_{\mathcal{S}}F$. 
  \end{description}
  Then $\P$ contains an initial object, given by the trivial spanning poset; we define $\P^\circ$ to be the full subposet containing all nontrivial spanning posets.
\end{definition}
The poset $\P$ has depth $2k-2$, which we can see from the following analysis.

Define
\[\mathrm{size}(\mathcal{S}) = 2|\ob \mathcal{S}| - |\mathrm{atomic\ morphisms\
    in\ }\mathcal{S}|.\]
For any relation $\mathcal{S} < \mathcal{S}'$ we have
$\mathrm{size}(\mathcal{S}) < \mathrm{size}(\mathcal{S}')$.  The minimal size of
a minimal spanning poset is $2$, and the maximal size is $2k$.  Thus $N\P^\circ$
is a $2k-3$-dimensional simplcial set.

An $m$-simplex $\sigma\in D^n(R^k)_m$ (the $n$-th space in $D(R^k)$) can be
thought of as a functor $F_\sigma: [m]^n \rto \Sub_{R^k}$ satisfying certain
conditions (see \cite[Definition 3.9]{rognes}).  A \emph{pick site} (see
\cite[Definition 5.6]{rognes}) is an object $i\in [m]^n$ such that
\[\colim_{j<i} F_\sigma(j) \not\cong F_\sigma(i).\]
The \emph{submodule configuration} is the full subposet of $[m]^n$ given by the
initial object together with the pick sites.  The submodule configuration is
always a minimal spanning poset, by the definition of the
$S_\dotp$-construction.

Since we are working with free modules and ranks are well-defined, we can see
that submodule configurations are always finite of cardinality at most $k$.

\begin{definition}
  We can define a $\P$-valuation $\nu^{sm}$ on $D(R^k)$ by sending a simplex
  $\sigma$ to its submodule configuration.
\end{definition}

\begin{lemma}
  $\nu^{sm}$ is a well-defined $\P$-valuation on $D(R^k)$.  Moreover, it is
  equivariant with respect to the natural action of $GL_k(R)$ on $\Sub_{R^k}$ and
  on $D(R^k)$. 
\end{lemma}

From the above analysis we conclude the following:
\[D(R^k) \simeq \Sigma^\infty \Sigma N\P^\circ,\] and this equivalence is
$GL_k(R)$-equivariant.  We know that $\Sigma N \P^\circ$ is $2k-2$-dimensional.
Moreover, due to Rognes's analysis showing that
$D(R^k) \simeq \Sigma^\infty \Sigma (\Sigma^{-1}D^V(R^k))$, where
$\Sigma^{-1}D^V(R^k)$ is the common basis complex (\cite[Definition
14.5']{rognes}), we get the following conjecture:
\begin{conjecture}
$GL_k(R)$-equivariantly,
  \[\Sigma^{-1}D^V(R^k) \simeq \Sigma N\P^\circ.\]
\end{conjecture}
We know that this is true after applying $\Sigma^\infty$, and thus that these spaces have the same homology.  Moreover, when $k \geq 2$ these are simply-connected, so to show that these are weakly equivalent it should be enough to construct a $GL_k(R)$-equivariant map between them that induces an isomorphism on homoloty.

Moreover, given the examples above, we can conjecture the following:
\begin{conjecture}
  $\P^\circ$ is a wedge of $2k-3$-spheres. 
\end{conjecture}


\appendix

\section{On homotopy types of decomposition posets (by Alexander
  Kupers)} \label{app:kupers}

The purpose of this appendix is to provide a result relating the homotopy types of decomposition posets of objects in convenient addition categories (see Definition~\ref{def:convenient}) to that of their ordered variants, which are more commonly studied in the literature. We will also remark on the relationship between this result and some notions from \cite{galatius_kupers_randal-williams18}. 

\subsection{Comparison} Before stating the main theorem of this appendix, we require some definitions.

\begin{definition}
  A convenient addition category $\A$ \emph{has ranks} if there exists a monoid homomorphism
  \[\rk: (\ob \A/\mathrm{iso},\oplus) \rto \Z_{\geq 0},\]
  which we call \emph{rank}, such that every nonzero object in $\A$ has a nonzero rank.
\end{definition}

\begin{example}The convenient addition categories described in
Section~\ref{sec:applications} have ranks. In the case of $\FinSet$ the rank is
the size of the set, in the case of orthogonal decompositions we take the
dimensions of the vector space, and in the case of finitely-generated free modules over a Dedekind domain we take the rank of a module.\end{example}

\begin{definition} A \emph{chain of length $n$} in a poset $\P$ is a sequence $p_0 < \cdots < p_n$. A \emph{longest chain} is a chain $p_0 < \cdots < p_n$ is a chain such that there does not exist a chain $q_0 < \cdots < q_m$ with $m > n$. A poset is \emph{bounded} if a longest chain exists.
\end{definition}



The main theorem of this appendix is the following:

\begin{theorem} \label{thm:deorder} Let $\A$ be a convenient decomposition
  category which has ranks, and moreover assume that for all nonzero $A\in \A$,
  the length of the longest chain in $\Decomp_A$ is $\rk(A)-1$.  Then
  $(\Decomp_A^{ord})^\circ$ is homotopy equivalent to a wedge of
  $(\rk(A)-2)$-spheres if and only if $\Decomp^\circ_A$ is homotopy equivalent to a
  wedge of $(\rk(A)-2)$-spheres.
\end{theorem}

In many situations $(\Decomp_A^{ord})^\circ$ has been studied before, and this theorem allows us to obtain information about $\Decomp^\circ_A$. The main tool for its proof is the following result due to Mirzaii and van der Kallen \cite[Theorem 3.8]{mirzaii_vanderkallen}:

\begin{lemma}[Mirzaii--van der Kallen] \label{lem:LvdK} Suppose that $\bA$ and $\bB$ are bounded posets and we have a map of posets
  \[F:\bA \rto \bB^\op.\] 
Suppose that in addition we have $n\in \Z$ and a strictly increasing function $t: \bB \rto \Z$ such that $\bB_{<b}$ is $(t(b)-2)$-connected and $\bA_{\leq b} = \{a\in \bA \,|\, b \leq F(a)\}$ is $(n - t(b)-1)$-connected. Then $\bA$ is $(n-1)$-connected if and only if $\bB$ is $(n-1)$-connected.
\end{lemma}


\begin{proof}[Proof of Theorem~\ref{thm:deorder}]
  First, we claim that the length of a
  longest chain in $\Decomp_A$ is the same as the length of a longest chain in
  $\Decomp_A^{ord}$. In one direction, note that any chain in $\Decomp_A^{ord}$ maps to a chain in
  $\Decomp_A$, so the length of a longest chain in $\Decomp_A$ is at least the length of a longest chain in $\Decomp_A^{ord}$. In the other direction, the elements in any chain
  in $\Decomp_A$ can be given an ordering such that all maps in the
  chain are order-preserving, so there is a chain in $\Decomp_A^{ord}$
  which maps to this chain under the forgetful functor. 

  Since the nerves of $\Decomp_A$ and $\Decomp_A^{ord}$ are $(\rk(A)-2)$-dimensional,
  they are both homotopy equivalent to a wedge of $(\rk(A)-2)$-spheres if and only if
  they are $(\rk(A)-3)$-connected.

  We proceed by induction on $\rk(A)$. For the initial cases, if $\rk(A)<2$ then $\Decomp^\circ_A$ and $(\Decomp_A^{ord})^\circ$ are empty since $A$ can not be decomposed into two non-initial objects. For the induction step, we will apply Lemma~\ref{lem:LvdK} with $n=\rk(A)-2$, $\bA = (\Decomp_A^{ord})^\circ$, $\bB = (\Decomp_A^\op)^\circ$ and $F$ being the forgetful functor
  $\Decomp_A^{ord} \rto \Decomp_A$. 
  For $p\in \Decomp^\circ_A$ we define $t(p) = \rk(A)-|p|$, with $|p|$ the number of elements in $p$. On the one hand, then $(\Decomp^\circ_A)^\op_{< p} = (\Decomp^\circ_A)_{> p}$ is given by
  all refinements of $p$ as decompositions of $A$.  If
  $p = \{A_1,\ldots,A_{|p|}\}$ then this is isomorphic to the join
  $\Decomp^\circ_{A_1} * \cdots * \Decomp^\circ_{A_n}$.  Since
  $\rk(A_j) < \rk(A)$ for all $j$, the induction hypothesis applies to each term.
  By induction, $(\Decomp^\circ_A)_{>p}$ is therefore a wedge of spheres of dimension
  \[-1 + \sum_{i=1}^{|p|} (\rk(A_i)-2+1) = \rk(A)-|p|-1 = t(p)-1,\] and is
  therefore $(t(p)-2)$-connected.  On the other hand, the poset $(\Decomp_A^{ord})^\circ_{\leq p}$
  is the poset of ordered partitions of the set $\{1,\ldots,|p|\}$ into at least
  two non-empty subsets, ordered by refinement.  This is the boundary of a
  permutahedron of order $|p|$, which is a $(|p|-2)$-sphere so is
  \[|p|-3 = (\rk(A)-2)- (\rk(A)-|p|)-1 = (\rk(A)-2) - t(p)-1\]
  connected.  Applying Lemma~\ref{lem:LvdK} gives the desired result.
\end{proof}

\begin{remark}Replacing \cite[Theorem 3.8]{mirzaii_vanderkallen} with the slightly stronger \cite[Corollary 2.2]{looijenga_vanderkallen}, we obtain further that $\widetilde{H}_{\rk(A)-2}(\Decomp^\circ_A) \rto \widetilde{H}_{\rk(A)-2}((\Decomp^{ord}_A)^\circ)$ is surjective.
\end{remark}

\subsection{Relation to \cite{galatius_kupers_randal-williams18}} Let us remark on the relationship of the previous result to \cite{galatius_kupers_randal-williams18} (see also \cite{rw-mo}); all references refer to this paper. If $\A$ is a convenient addition category with ranks then its groupoid core $i\A$ is a symmetric monoidal groupoid with rank functor that satisfies Assumptions 17.1 and 17.2. Section 17.1 associates to it a non-unital $E_\infty$-algebra $\mathbf{T}$ in the functor category $\mathrm{\Fun}(i\A,\mathrm{sSet})$ with Day convolution symmetric monoidal structure, so that $\mathbf{T}(\A) = \varnothing$ if $A \cong \varnothing$ and $\mathbf{T}(A) = \ast$ otherwise. This has derived $E_k$-indecomposables $Q^{E_k}_\mathbb{L}(\mathbf{T}) \in \mathrm{\Fun}(i\A,\mathrm{sSet}_*)$ for $k=1,2,\ldots,\infty$. By Section 17.5 and Section 17.4 these can be computed as 
\[Q^{E_1}_\mathbb{L}(\mathbf{T})(A) \simeq \Sigma' S^{E_1}(A) \qquad \text{and} \qquad Q^{E_\infty}_\mathbb{L}(\mathbf{T})(A) \simeq \Sigma' S^{E_\infty}(A),\] where $S^{E_1}(A)$ is the $E_1$-splitting complex of Definition 17.9 and $S^{E_\infty}(A)$ is the $E_\infty$-splitting complex of Definition 17.18. As $S^{E_1}(A)$ is equivalent to the geometric realisation of $(\Decomp^{ord}_A)^\circ$  and $S^{E_\infty}(A)$ is equivalent to the geometric realisation of $\Decomp_A^\circ$, interpreted in this language, Theorem~\ref{thm:deorder} above says that under the stated hypotheses $Q^{E_\infty}_\mathbb{L}(\mathbf{T})(A)$ is $(\rk(A)-2)$-connected if and only if $Q^{E_1}_\mathbb{L}(\mathbf{T})(A)$ is. This can also be deduced by ``transferring vanishing lines up'' using Theorem 14.4.

\bibliographystyle{alpha}
\bibliography{CZ}

\end{document}